\documentclass[12pt,reqno]{amsart}

\usepackage{amssymb}
\usepackage{amscd}
\usepackage{amsfonts}
\usepackage{setspace}
\usepackage{version}
\usepackage{mathrsfs}

\usepackage{graphicx}

\newtheorem{theorem}{Theorem}[section]
\newtheorem{lemma}[theorem]{Lemma}
\newtheorem{proposition}[theorem]{Proposition}

\newtheorem*{weyl-restrict-rpt}{Proposition \ref{weyl-restricted}}
\newtheorem*{weyl-restricted-large-rpt}{Proposition \ref{weyl-restricted-large}}
\newtheorem*{add-digit-hamming-rpt}{Proposition \ref{add-digit-hamming}}

\theoremstyle{definition}

\newtheorem{definition}[theorem]{Definition}

\renewcommand{\leq}{\leqslant}
\renewcommand{\geq}{\geqslant}

\newcommand\Supp{\operatorname{Supp}}

\def\F{\mathbf{F}}
\def\R{\mathbf{R}}
\def\C{\mathbf{C}}
\def\Z{\mathbf{Z}}
\def\E{\mathbf{E}}
\def\P{\mathbf{P}}

\def\N{\mathbf{N}}

\def\x{\mathbf{x}}
\def\y{\mathbf{y}}
\def\z{\mathbf{z}}
\newcommand\meas{\operatorname{meas}}

\def\dd{\operatorname{d}}

\def\w{\operatorname{w}}
\newcommand{\zn}[1]{[0, #1)}

\newcommand{\ltrans}{\phi}
\def\eps{\varepsilon}

\newcommand{\md}[1]{\ensuremath{(\operatorname{mod}\, #1)}}

\numberwithin{equation}{section}

\begin{document}

\title[Waring with restricted digits]{Waring's problem with restricted digits}
\author{Ben Green}\email{ben.green@maths.ox.ac.uk}



\begin{abstract}
Let $k \geq 2$  and $b \geq 3$ be integers, and suppose that $d_1, d_2 \in \{0,1,\dots, b - 1\}$ are distinct and coprime. Let $\mathcal{S}$ be the set of non-negative integers, all of whose digits in base $b$ are either $d_1$ or $d_2$. Then every sufficiently large integer is a sum of at most $b^{160 k^2}$ numbers of the form $x^k$, $x \in \mathcal{S}$.
\end{abstract}
\thanks{The author gratefully acknowledges the support of the Simons Foundation (Simons Investigator grant 376201).} 
\maketitle

\tableofcontents

\section{Introduction}

Let $k \geq 2$ be an integer. One of the most celebrated results in additive number theory is Hilbert's theorem that the $k$th powers are an asymptotic basis of finite order. That is, there is some $s$ such that every sufficiently large natural number can be written as a sum of at most $s$ $k$th powers of natural numbers.

One may ask whether a similar result holds if one passes to a subset $\{ x^k : x \in \mathcal{S}\}$ of the full set of $k$th powers. This has been established in various cases, for instance when $\mathcal{S}$ is the set of primes (the so-called Waring-Goldbach problem \cite{kumchev-tolev}), the set of smooth numbers with suitable parameters \cite{drappeau-shao}, the set of integers such that the sum of digits in base $b$ lies in some fixed residue class modulo $m$ \cite{thuswaldner-tichy}, random sets with $\P(s \in \mathcal{S}) = s^{c - 1}$ for some $c > 0$ \cite{vu,wooley-thin}, or \emph{all} sets with suitably large density \cite{salmensuu}.

Our main result in this paper is that a statement of this type holds when $\mathcal{S}$ is the set of integers whose base $b$ expansion contains just two different (fixed) digits.

\begin{theorem}\label{mainthm}
Let $k \geq 2$ and $b \geq 3$ be integers, and suppose that $d_1, d_2 \in \{0,1,\dots, b - 1\}$ are distinct and coprime. Let $\mathcal{S}$ be the set of non-negative integers, all of whose digits in base $b$ are either $d_1$ or $d_2$. Then every sufficiently large integer is a sum of at most $b^{160 k^2}$ numbers of the form $x^k$, $x \in \mathcal{S}$.
\end{theorem}

\emph{Remarks.} While the basic form of the bound is the best the method gives, the constant $160$ could certainly be reduced, especially for large values of $b$; I have not tried to optimise it. The restriction to $b \geq 3$ is helpful at certain points in the argument. Of course, the case $b = 2$ (in which case we must have $\{d_1, d_2\} = \{0,1\}$) corresponds to the classical Waring problem, for which much better bounds are known.

Although Theorem \ref{mainthm} seems to be new, one should certainly mention in this context the interesting work of Biggs \cite{biggs-1, biggs-2} and Biggs and Brandes \cite{biggs-brandes}, who showed that, for some $s$, every sufficiently large integer is a sum of at most $s$ numbers of the form $x^k$, $x \in \mathcal{S}$, and one further $k$th power. (In their work $b$ is taken to be prime and larger than $k$.)

This paper is completely independent of the work of Biggs and Brandes, but it seems plausible that by combining their methods with ours one could significantly reduce the quantity $b^{160k^2}$ in Theorem \ref{mainthm}, at least for prime $b$. 

Finally, we note that sets of integers whose digits in some base are restricted to some set are often called \emph{ellipsephic}, a term coined by Mauduit, as explained in \cite{biggs-1,biggs-2}.\vspace*{8pt}

\emph{Notation.} If $x \in \R$, we write $\Vert x \Vert$ for the distance from $x$ to the nearest integer. The only other time we use the double vertical line symbol is for certain box norms $\Vert \cdot \Vert_{\Box}$ which occur in Appendix \ref{appA}. There seems little danger of confusion so we do not resort to more cumbersome notations such as $\Vert x \Vert_{\R/\Z}$. Write $e(x) = e^{2 \pi i x}$.

If $X$ is a finite set and $f : X \rightarrow \C$ is a function then we write $\E_{x \in X}f(x) = \frac{1}{|X|} \sum_{x \in X} f(x)$ .

All intervals will be discrete. Thus $[A,B]$ denotes the set of all \emph{integers} $x$ with $A \leq x \leq B$ (and here $A,B$ need not be integers). We will frequently encounter the discrete interval $\zn{m}$, for positive integer $m$, which is the same thing as the set $ \{0,1,\dots, m - 1\}$. Note carefully that at some points in Section \ref{sec5}, the notation $[m_1, m_2]$ will also refer to the lowest common multiple of two integers $m_1, m_2$.

Throughout the paper we will fix a base $b \geq 3$, an exponent $k \geq 2$ and distinct coprime digits $d_1,d_2 \in  \zn{b}$. Denote by $\mathcal{S}$ the set of all non-negative integers $x$, all of whose digits in base $b$ are $d_1$ or $d_2$. We include $0$ in $\mathcal{S}$. Write $\mathcal{S}^k := \{x^k : x\in \mathcal{S}\}$. Note that $\mathcal{S}^k$ might more usually refer to the $k$-fold product set of $\mathcal{S}$ with itself, but we have no use for that concept here.

We will reserve the letter $n$ for a variable natural number, which we often assume is sufficiently large, and which it is usually convenient to take to be divisible by $k$. We always write $N = b^n$, so $\zn{N}$ is precisely the set of non-negative integers with at most $n$ digits in base $b$.

If $n$ is a natural number, we define the map $L_b : \{0,1\}^{\zn{n}} \rightarrow \Z$ by
\begin{equation}\label{L-def} L_b(\x) := \sum_{i \in \zn{n}} x_i b^{i},\end{equation} where $\x = (x_i)_{i \in [0,n)}$.
Although this map depends on $n$, we will not indicate this explicitly, since the underlying $n$ will be clear from context.
Then 
\begin{equation}\label{lb-form} \frac{d_1(b^n - 1)}{b - 1} + (d_2 - d_1) L_b(\x)\end{equation} is the number whose base $b$ expansion has $b^i$-digit equal to $d_1$ if $x_i = 0$, and $d_2$ if $x_i = 1$.\vspace*{8pt}

\emph{Acknowledgements.} I thank Zach Hunter and Sarah Peluse for comments on the first version of the manuscript.

\section{An outline of the argument}\label{outline-sec}

Unsurprisingly, given its pre-eminence in work on Waring's problem, the basic mode of attack is the Hardy-Littlewood circle method. Let $n \in \N$, set $N = b^n$ and consider the subset of $\mathcal{S}$ consisting of integers with precisely $n$ digits. This is a set of size $2^n$. Denote by $\mu_n$ the normalised probability measure on the set of $k$th powers of the elements of this set. That is, $\mu_n(m) = 2^{-n}$ if $m = (\sum_{i \in \zn{n}} x_i b^i)^k$ with all $x_i \in \{d_1, d_2\}$ for all $i$, and $\mu_n(m) = 0$ otherwise. The Fourier transform $\widehat{\mu_n}(\theta) := \sum_{m \in \Z} \mu_n(m) e(m \theta)$ is then a normalised version of what is usually called the exponential sum or Weyl-type sum, and as expected for an application of the circle method, it plays a central role in our paper.

Our main technical result is the following, which might be called a log-free Weyl-type estimate for $k$th powers with restricted digits.

\begin{proposition}\label{weyl-restricted}
Suppose that $k \geq 2$ and $b \geq 3$. Set $B := b^{6k^2}$. Suppose that $\delta \in (0,1)$ and that $k \mid n$. Suppose that $|\widehat{\mu_n}(\theta)| \geq \delta$, and that $N \geq (2/\delta)^{B}$, where $N := b^n$. Then there is a positive integer $q \leq (2/\delta)^{B}$ such that $\Vert \theta q \Vert \leq (2/\delta)^{B}N^{-k}$.
\end{proposition}

\emph{Remarks.} If $\mu_n$ is replaced by the normalised counting measure on $k$th powers less than $N$ without any digital restriction, a similar estimate is true and is very closely related to Weyl's inequality. The most standard proof of Weyl's inequality such as \cite[Lemma 2.4]{vaughan-hl}, however, results in some extra factors of $N^{o(1)}$ (from the divisor bound). ``Log-free'' versions may be obtained by combining the standard result with major arc estimates as discussed, for example, in \cite{wooley-freeman}, or by modifying the standard proof of Weyl's inequality to focus on this goal rather than on the quality of the exponents, as done in \cite[Section 4]{green-tao-linearprimes}. Our treatment here is most closely related to this latter approach.

Although we will only give a detailed proof of Proposition \ref{weyl-restricted} in the case that $\mu_n$ is the measure on $k$th powers of integers with just two fixed digits, similar arguments ought to give a more general result in which the digits are restricted to an arbitrary subset of $\{0,1,\dots, b - 1\}$ of size at least $2$. This would be of interest if one wanted to obtain an asymptotic formula in Theorem \ref{mainthm}, with more general digital restrictions of this type.\vspace*{8pt}

Experts will consider it a standard observation that Proposition \ref{weyl-restricted} implies that $\mathcal{S}^k$ is an asymptotic basis of some finite order $s$. Roughly, this is because one can use it to obtain a moment estimate $\sum_x \mu_n^{(t)}(x)^2  = \int^1_0 |\widehat{\mu_n}(\theta)|^{2t} d\theta \ll N^{-k}$ for a suitably large $t$. Here, $\mu_n^{(t)}$ denotes the $t$-fold convolution power of $\mu_n$; see immediately after \eqref{t-choice} for full details. Cauchy-Schwarz then implies that the $t$-fold sumset $t \mathcal{S}^k$ has positive density in an interval of length $\gg N^k$, whereupon methods of additive combinatorics can be used to conclude. 

However, by itself this kind of argument leads to $s$ having a double-exponential dependence on $k$. The reason is that Proposition \ref{weyl-restricted} is not very effective in the regime $\delta \approx 1$. It is possible that the proof could be adapted so as to be more efficient in this range, but this seems nontrivial. Instead we provide, in Section \ref{very-large-values}, a separate argument which is at first sight crude, but turns out to be more efficient for this task. This gives the following result.

\begin{proposition}\label{weyl-restricted-large}
Let $n \in \N$ and let $N = b^n$. Suppose that $n \geq k$. Then the measure of all $\theta \in \R/\Z$ such that $|\widehat{\mu_n}(\theta)| \geq 1 - \frac{1}{4}b^{-3k^2}$ is bounded above by $2b^{k^2} N^{-k}$.
\end{proposition}
In fact, we obtain a characterisation of these values of $\theta$, much as in Proposition \ref{weyl-restricted}: see Section \ref{very-large-values} for the detailed statement and proof.

Details of how to estimate the moment $\int^1_0 |\widehat{\mu_n}(\theta)|^{2t} d \theta$ using Propositions \ref{weyl-restricted} and \ref{weyl-restricted-large}, and of the subsequent additive combinatorics arguments leading to the proof of Theorem \ref{mainthm}, may be found in Section \ref{basis-from-weyl}.

This leaves the task of proving Proposition \ref{weyl-restricted}, which forms the bulk of the paper, and is where the less standard ideas are required. For the purposes of this overview, we mostly consider the case $k = 2$, and for definiteness set $\{d_1, d_2\} = \{0,1\}$. \vspace*{8pt}

\emph{Decoupling.}  The first step is a kind of decoupling. Recall the definitions of the maps $L_b$ (see \eqref{L-def}). The idea is to split the variables $\x = (x_i)_{i \in \zn{n}}$ into the even variables $\y = (x_{2i})_{i \in \zn{n/2}}$ and the odd variables $\z = (x_{2i+1})_{i \in \zn{n/2}}$, assuming that $n$ is even for this discussion. We have $L_b(\x) = L_{b^2}(\y) + b L_{b^2}(\z)$. Here, there is a slight abuse of notation in that $L_b$ is defined on vectors of length $n$, whilst $L_{b^2}$ is defined on vectors of length $n/2$. We then have
\begin{align*} \widehat{\mu_n}(\theta) & = \E_{\x \in \{0,1\}^{\zn{n}}} e(\theta L_b(\x)^2) \\ & = \E_{\y, \z \in \{0,1\}^{\zn{n/2}}} e\Big( \theta \big(L_{b^2}(\y) + b L_{b^2}(\z)^2\big)\Big) \\ & = \E_{\y, \z \in \{0,1\}^{\zn{n/2}}} \Psi(\y) \Psi'(\z) e\big( 2b \theta L_{b^2}(\y) L_{b^2}(\z)\big),\end{align*}
where $\Psi(\y) = e(\theta L_{b^2}(\y)^2)$ and $\Psi'(\z) = e(b^2 \theta L_{b^2}(\z)^2)$, but the precise form of these functions is not important in what follows. 
By two applications of the Cauchy-Schwarz inequality (see Appendix \ref{appA} for a general statement), we may eliminate the $\Psi$ and $\Psi'$ terms, each of which depends on just one of $\y, \z$. Assuming, as in the statement of Proposition \ref{weyl-restricted}, that $|\widehat{\mu_n}(\theta)| \geq \delta$, we obtain
\begin{align*}  \delta^4 \leq \E_{\y, \z, \y', \z'  \in \{0,1\}^{\zn{n/2}}}  & e\big(2b \theta \big( L_{b^2}(\y) L_{b^2}(\z) - L_{b^2}(\y') L_{b^2}(\z)  -\\ & -  L_{b^2}(\y) L_{b^2}(\z') + L_{b^2}(\y') L_{b^2}(\z')\big)\big).\end{align*}
We remove the expectation over the dashed variables, that is to say there is some choice of $\y', \z'$ for which the remaining average over $\y, \z$ is at least $\delta^4$. For simplicity of discussion, suppose that $\y' = \z'= 0$ is such a choice; then 
\begin{equation}\label{decoupled} \delta^4 \leq \E_{\y, \z \in \{0,1\}^{\zn{n/2}}}  e\big( 2b \theta  L_{b^2}(\y) L_{b^2}(\z)\big).\end{equation}
At the expense of replacing $\delta$ by $\delta^4$, we have replaced the quadratic form $L_b(\x)^2$ by a product of two linear forms in disjoint variables, which is a far more flexible object to work with.
I remark that I obtained this idea from the proof of \cite[Theorem 4.3]{ctv}, which uses a very similar method.

Now, for fixed $\z$ the average over $\y$ in \eqref{decoupled} can be estimated fairly explicitly. The conclusion is that for $\gg \delta^4 2^{n/2}$ values of $\z$, $2b \theta L_{b^2}(\z)$ has $\ll \log(1/\delta)$ non-zero base $b$ digits, among the first $n$ digits after the radix point. Here, we use the \emph{centred} base $b$ expansion in which digits lie in $(-\frac{b}{2}, \frac{b}{2}]$, discussed in more detail in Section \ref{decoupling-sec}.
\vspace*{8pt}

\emph{Additive expansion.} The output of the decoupling step is an assertion to the effect that, for $m$ in a somewhat large set $\mathscr{M} \subset \{1,\dots, N\}$, $\theta m$ has very few non-zero digits in base $b$ among the first $n$ after the radix point. The set $\mathscr{M}$ is the set of $2b L_{b^2}(\z)$ for $\gg \delta^4 2^{n/2}$ values of $\z \in \{0,1\}^{[0, n/2)}$, and so has size $\sim N^{(\log 2)/2\log b}$ which, though `somewhat large', is unfortunately appreciably smaller than $N$.

The next step of the argument is to show that the sum of a few copies of $\mathscr{M}$ is a considerably larger set, of size close to $N$. In fact, in the case $k = 2$ under discussion, $b^2-1$ copies will do. This follows straightforwardly from the following result from the literature.

\begin{theorem}\label{bkmp-thm}
Let $r, n \in \N$. Suppose that $A_1,\dots, A_r \subseteq \{0,1\}^n$ are sets with densities $\alpha_1,\dots, \alpha_r$. Then $A_1 + \cdots + A_r$ has density at least $(\alpha_1 \cdots \alpha_r)^{\gamma}$ in $\{0,1,\dots, r\}^n$, where $\gamma := r^{-1} \log_2(r+1)$.
\end{theorem}
This theorem, which came from the study of Cantor-type sets in the 1970s and 1980s, seems not to be well-known in modern-day additive combinatorics. The result has a somewhat complicated history, with contributions by no fewer than 10 authors, and I am unsure exactly how to attribute it. For comments and references pertinent to this, see Appendix \ref{appB}.

 We remark that for $k > 2$ a considerably more elaborate argument is required at this point, and this occupies the bulk of Section \ref{sec5}.

The conclusion is that $\theta m$ has $\ll \log(1/\delta)$ nonzero base $b$ digits among the first $n$ after the radix point, for all $m$ in a set $\mathscr{M}' \subset \{1,\dots, N\}$ of size $\gg \delta^{C} N$.\vspace*{8pt}

\emph{From digits to diophantine.} In the final step of the argument we extract the required diophantine conclusion (that is, the conclusion of Proposition \ref{weyl-restricted}) from the digital condition just obtained. The main ingredient is a result on the additive structure of sets with few nonzero digits, which may potentially have other uses. Recall that if $A$ is a set of integers then $E(A)$, the additive energy of $A$, is the number of quadruples $(a_1, a_2, a_3, a_4) \in A \times A \times A \times A$ with $a_1 + a_2 = a_3 + a_4$.

\begin{proposition}\label{add-digit-hamming}
Let $r \in \Z_{\geq 0}$. Suppose that $A \subset \Z$ is a finite set, all of whose elements have at most $r$ nonzero digits in their centred base $b$ expansion. Then $E(A) \leq (2b)^{4r} |A|^{2}$.
\end{proposition}

The proof of this involves passing to a quadripartite formulation (that is, with four potentially different sets $A_1 ,A_2, A_3, A_4$, and also allowing for the possibility of a `carry' in the additive quadruples) and an inductive argument. 

The final deduction of Proposition \ref{weyl-restricted} uses this and some fibring arguments. This, and the proof of Proposition \ref{add-digit-hamming}, may be found in Section \ref{sec6}.

\section{Reduction to a log-free Weyl-type estimate}\label{basis-from-weyl}

In this section we show that our main result, Theorem \ref{mainthm} follows from the log-free Weyl-type estimate, Proposition \ref{weyl-restricted}. We begin by stating two results about growth under set addition. The first is a theorem of Nathanson and S\'ark\"ozy.

\begin{theorem}\label{ns-thm}
Let $X \in \N$ and $r \in \N$. Suppose that $A \subset \{1,\dots, X\}$ is a set of size $\geq 1 + X/r$. Then there is an arithmetic progression of common difference $d$, $1 \leq d \leq r - 1$ and length at least $\lfloor X/2r^2\rfloor$ contained in $4r A$.
\end{theorem}
\begin{proof}
In \cite[Theorem 1]{nathanson-sarkozy}, take $h = 2r$, $z = \lfloor X/2r^2\rfloor$; the result is then easily verified.
\end{proof}

The second result we will need is a simple but slightly fiddly lemma on repeated addition of discrete intervals.

\begin{lemma}\label{interval-sum}
Let $X \geq 1$ be real and suppose that $I \subset \zn{X}$ is a discrete interval of length $L \geq 2$. Set $\eta := L/X$. Let $K \geq 4$ be a parameter.  Then $\bigcup_{j \leq \lceil 2K/\eta^2\rceil} j I$ contains the discrete interval $[\frac{4}{\eta} X, \frac{K}{\eta} X]$.
\end{lemma}
\begin{proof}
Write $I = [x_0, x_0 + L-1]$, where $x_0 \in \Z_{\geq 0}$. Then $jI = [jx_0, jx_0 + j(L-1)]$. Note that if $j \geq x_0/(L-1)$, we have $jx_0 + j(L-1) \geq (j+1) x_0$, and so the interval $(j+1) I$ overlaps the interval $jI$. Therefore if we set $j_0 := \lceil x_0/(L-1) \rceil$, for any $j_1 \geq j_0$ the union $I^* := \bigcup_{j_0 \leq j \leq j_1} jI$ is a discrete interval. Set $j_1 := \lceil 2K/\eta^2 \rceil$. We have
\[ \min I^* = j_0x_0  \leq \big\lceil \frac{X}{L-1}\big\rceil X \leq \big\lceil \frac{2X}{L} \big\rceil X \leq \frac{4X^2}{L} = \frac{4}{\eta} X,\] and
\[ \max I^* \geq j_1 (L-1) \geq \frac{2K}{\eta^2} \frac{L}{2} = \frac{K}{\eta} X.\] 
This concludes the proof.
\end{proof}

\begin{proof}[Proof of Theorem \ref{mainthm}, assuming Proposition \ref{weyl-restricted}] Let $n$ be some large multiple of $k$ and consider the measure $\mu_n$ as described in Section \ref{outline-sec}. Thus $\mu_n$ is supported on $\mathcal{S}^k \cap [0, N^k)$, where $N = b^n$. Set \begin{equation}\label{t-choice} t := 8b^{9k^2},\end{equation}  and write $\mu_n^{(t)}$ for the $t$-fold convolution power of $\mu_n$, that is to say $\mu_n^{(t)}(x) = \sum_{x_1 + \cdots + x_{t} = x} \mu_n(x_1) \cdots \mu_n(x_t)$. Then $\widehat{\mu_n^{(t)}} = ( \widehat{\mu_n})^t$ and so by Parseval's identity and the layer-cake representation

\begin{align}\nonumber \sum_x & \mu_n^{(t)}(x)^2  = \int^1_0 |\widehat{\mu_n}(\theta)|^{2t} d\theta \\ & = 2t \int^1_0 \delta^{2t - 1} \meas \{ \theta : |\widehat{\mu_n}(\theta)| \geq \delta \} d \delta = 2t (I_1 + I_2 + I_3),\label{2t-moment}\end{align}
where $I_1, I_2, I_3$ are the integrals over ranges $[0, 2N^{-1/B}]$, $[2N^{-1/B}, 1 - c]$ and $[1 - c, 1]$ respectively, with $c := \frac{1}{4} b^{-3k^2}$, $B  = b^{6k^2}$ (as in Proposition \ref{weyl-restricted}) and $\meas$ is the Lebesgue measure on the circle $\R/\Z$. We have, for $N$ large,
\[ I_1 \leq (2N^{-1/B})^{2t - 1}  < N^{-k}.\]

To bound $I_2$ we use Proposition \ref{weyl-restricted}, which tells us that the set $\{ \theta \in \R/\Z : |\widehat{\mu}_n(\theta)| \geq \delta\}$ is contained in the set $\{ \theta \in \R/\Z : \Vert \theta q \Vert \leq (2/\delta)^B N^{-k} \; \mbox{for some positive} \; q \leq (2/\delta)^B\}$ , and so $\meas\{ \theta : |\widehat{\mu_n}(\theta)| \geq \delta \} \leq 2(2/\delta)^{2B} N^{-k}$. Since $2t - 1 - 2B \geq t$, we therefore have
\[ I_2  \leq 2N^{-k} \int^{1 - c}_0 \delta^{2t-1}(2/\delta)^{2B} d\delta \leq 2N^{-k} (1 - c)^t  2^{2B} < N^{-k}.\] 
For the last inequality we used the fact that $t = 2B/c$ and so $(1 - c)^t  \leq e^{-2B}$.

Finally, to bound $I_3$ we use Proposition \ref{weyl-restricted-large}, which immediately implies that
\[ I_3 \leq 2b^{k^2} N^{-k} .  \]
Substituting these bounds for $I_1, I_2$ and $I_3$ into \eqref{2t-moment}, we obtain that for $N$ sufficiently large $\sum_x \mu_n^{(t)}(x)^2 \leq 4t b^{k^2} N^{-k} = 32b^{10k^2} N^{-k}$. On the other hand, it follows by Cauchy-Schwarz and the fact that  $\sum_x \mu_n^{(t)}(x) = 1$ that $1 \leq |\Supp(\mu_n^{(t)})| \sum_x \mu_n^{(t)}(x)^2$, and so $|\Supp(\mu_n^{(t)})| \geq 2^{-5} b^{-10k^2}N^k$. Thus, since $\mu_n^{(t)}$ is supported on the $t$-fold sumset of $\mathcal{S}^k \cap \zn{N^k}$, we see that $|t \mathcal{S}^k \cap \zn{tN^k}| \geq 2^{-5} b^{-10k^2}N^k$. Applying Theorem \ref{ns-thm} with $X = tN^k$ and $r = 2^8 b^{19k^2}$, we see that $4rt \mathcal{S}^k \cap \zn{4rt N^k}$ contains an arithmetic progression $P$ of common difference $< r$ and length $|P| \geq L := 2^{-15} b^{-29k^2} N^k$.

Since $d_1^k$ and $d_2^k$ are coprime, every number greater than or equal to $(d_1^k - 1)(d_2^k - 1) < b^{2k} < r$ is a non-negative integer combination of these numbers. Therefore it is certainly the case that $2r \mathcal{S}^k$ contains $[r, 2r)$. Since the common difference of $P$ is less than $r$, $P + [r, 2r)$ contains a discrete interval $I$ of length $\geq L$. This interval is therefore contained in $(4rt + 2r) \mathcal{S}^k \subset 8rt \mathcal{S}^k$. Note that by construction $I \subset \zn{8rt N^k}$. 

Apply Lemma \ref{interval-sum}, taking $X = X(n) = 8rt N^k$, $\eta = \frac{L}{X} = 2^{-29} b^{-57k^2}$, and $K = 4b^{k^2}$. Since $\mathcal{S}$ contains $0$, we see that $\lfloor 2K/\eta^2\rfloor 8rt \mathcal{S}^k = 2^{75}b^{142k^2} \mathcal{S}^k$ contains the interval $I_n := [\frac{4}{\eta} X(n), \frac{K}{\eta} X(n)]$. Remember that here $n$ is any sufficiently large multiple of $k$. By the choice of $K$, $\frac{K}{\eta} X(n) = \frac{4}{\eta} X(n+k)$, and so these intervals overlap. Thus $\bigcup_{n} I_n$ consists of all sufficiently large integers, and hence so does $2^{75}b^{142k^2} \mathcal{S}^k$. Finally, one may note that $2^{75} < b^{12k^2}$ for $b \geq 3$ and $k \geq 2$.
\end{proof}

\section{Very large values of the Fourier transform}\label{very-large-values}

In this section we establish Proposition \ref{weyl-restricted-large}. We will in fact establish the following more precise result.

\begin{proposition}\label{prop41}
Let $n \in \N$ and let $N = b^n$. Suppose that $n \geq k$. Let $\theta \in \R/\Z$. Suppose that $|\widehat{\mu_n}(\theta)| \geq 1 - \frac{1}{4}b^{-3k^2}$. Then there is a positive integer $q \leq (2k!)b^{\frac{1}{2}k(k-1) + 1}$ such that $\Vert \theta q \Vert \leq (2k!)^{-1} b^{\frac{1}{2}k(k+1) - 1} N^{-k}$.
\end{proposition}

Proposition \ref{weyl-restricted-large} is a consequence of this and the observation that the measure of $\theta \in \R/\Z$ such that $\Vert \theta q \Vert\leq \eps$ for some positive integer $q \leq q_0$ is bounded above by $2\eps q_0$.

\begin{proof}[Proof of Proposition \ref{prop41}] Set $Q := 2k! b^{k (k - 1)/2 +1}$. Note that, since $2k! \leq 2^{k^2/2} \leq b^{k^2/2}$ for all $b,k \geq 2$, we have $Q \leq b^{k^2}$. By Dirichlet's theorem, there is some positive integer $q \leq Q$ and an $a$, coprime to $q$, such that $|\theta - a/q| \leq 1/qQ$. Set $\eta := \theta - a/q$, thus $|\eta| \leq 1/qQ$. There is a unique integer $j$ such that \begin{equation}\label{mid-sized} \frac{1}{2bq} < |(d_2 - d_1) k! b^j \eta | \leq \frac{1}{2q}.\end{equation}
Now if we had $j < k(k-1)/2$ then 
\[ |(d_2 - d_1) k! b^j \eta | \leq (b-1)k! b^{\frac{1}{2}k(k-1) - 1} |\eta| < k! b^{\frac{1}{2}k(k-1)}/qQ = 1/2bq,\] contrary to \eqref{mid-sized}. If $j > kn - k(k+1)/2$ then, using \eqref{mid-sized}, 
\[ \Vert \theta q\Vert = |\eta q| \leq  \big|2 (d_2 - d_1) k! b^j \big|^{-1} \leq (2k!)^{-1} b^{\frac{1}{2}k(k+1) - 1} N^{-k} ,\] in which case the conclusion of the proposition is satisfied.

Suppose, then, that $k(k-1)/2 \leq j \leq kn -k(k+1)/2$. Then there is a set $I \subset [0,n)$, $|I| = k$, such that $j = \sum_{i \in I} i$. As usual, write $\x = (x_i)_{i \in [0,n)}$. It is convenient to write $\x_I$ for the variables $x_i$, $i \in I$, and $\x_{[0,n) \setminus I}$ for the other variables.
For any fixed choice of $\x_{[0,n) \setminus I}$ we can write, setting $u := \frac{d_1(b^n - 1)}{b - 1}$,

\begin{align*} \big(u  + (d_2 - d_1) L_b(\x)\big)^k & = \big(u + (d_2 - d_1)\sum_{i \in [0, n)} x_i b^i \big)^k \\ & = (d_2 - d_1)k! b^j \prod_{i \in I} x_i + \sum_{i \in I} \psi_i(\x_{[0,n) \setminus I}; \x_I),\end{align*} for some functions $\psi_i$, where $\psi_i$ does not depend on $x_i$.
It follows that 
\begin{align*} & |\widehat{\mu}_n(\theta)|  = \big| \E_{\x \in \{0,1\}^{[0,n)}} e \big(\big(u  + (d_2 - d_1) L_b(\x)\big)^k \big)\big|  \leq \\ & \E_{\x_{[0,n) \setminus I} \in \{0,1\}^{[0,n) \setminus I}} \big| \E_{\x_I \in \{0,1\}^I} \prod_{i \in I} \Psi_i(\x_{[0,n) \setminus I}; \x_I) e \big( (d_2 - d_1) k! b^j \theta \prod_{i \in I} x_i\big) \big|,\end{align*} where $\Psi_i := e(\psi_i)$ is a 1-bounded function, not depending on $x_i$.
By Proposition \ref{box-gvn} (and the accompanying definition of Box norm, Definition \ref{defA1}) it follows that 
\[ |\widehat{\mu_n}(\theta)|^{2^k} \leq \E_{\x_I, \x'_I \in \{0,1\}^I} e\big((d_2 - d_1) k! b^j \theta \prod_{i \in I} (x_i - x'_i)\big).\] (The right-hand side here is automatically a non-negative real number).
On the right, we now bound all the terms trivially (by $1$) except for two: the term with $x_i = x'_i = 0$ for all $i \in I$, and the term with $x_i = 1$ and $x'_i = 0$ for all $i \in I$. This gives, using the inequality $2 - |1 + e(t)| = 4 \sin^2 \frac{\pi \Vert t \Vert}{2} \geq 4 \Vert t\Vert^2$,
\begin{align}\nonumber |\widehat{\mu_n}(\theta)|^{2^k} & \leq 1 - \frac{2}{4^k} + \frac{1}{4^k} |1 + e((d_2 - d_1)k! b^j \theta)|  \\  & \leq 1 - 2^{2 - 2k} \Vert (d_2 - d_1) k! b^j \theta \Vert^2.\label{mu-2k} \end{align} 
There are now two slightly different cases, according to whether or not $q \mid (d_2 - d_1) k! b^j a$. If this is the case, then by \eqref{mid-sized}
\[ \Vert (d_2 - d_1) k! b^j \theta\Vert = | (d_2 - d_1) k! b^j \eta | \geq 1/2bq.\] If, on the other hand, $q \nmid (d_2 - d_1) k! b^j a$ then by \eqref{mid-sized} we have
\[ \Vert (d_2 - d_1) k! b^j \theta \Vert \geq \frac{1}{q} - | (d_2 - d_1) k! b^j \eta | \geq \frac{1}{2q}.\] 
In both cases, $\Vert (d_2 - d_1) k! b^j \theta\Vert \geq 1/2bQ = (4k!)^{-1} b^{-2-\frac{1}{2}k(k-1)}$. It follows from \eqref{mu-2k} that 
\[ |\widehat{\mu_n}(\theta)| \leq \big(1 - 2^{2 - 2k} (4k!)^{-2} b^{-4 - k(k-1)}\big)^{1/2^k} < 1 - \frac{1}{4}b^{-3k^2}, \] that is to say the hypothesis of the proposition is not satisfied. Here, the second inequality follows from the Bernoulli inequality $(1 - x)^{1/2^k} \leq 1 - x/2^k$ and the crude bounds $k! \leq b^{k^2/4}$, $2^{3k} \leq b^{2k}$, both valid for $b \geq 3$ and $k \geq 2$.
\end{proof}

\section{Decoupling}\label{decoupling-sec}

We turn now to the somewhat lengthy task of proving Proposition \ref{weyl-restricted}. In this section we give the details of what we called the decoupling argument in the outline of Section \ref{outline-sec}.  The main result of the section is Proposition \ref{sec3-main} below. We begin with a definition.

\begin{definition}\label{wn-tilde-def}
Let $\alpha \in \R/\Z$. Then we define
\begin{equation}\label{w-def} \tilde\w_n(\alpha) := \sum_{i \in \zn{n}} \Vert \alpha b^i \Vert^2.\end{equation}
\end{definition}
The reason for the notation is that $\tilde\w_n(\alpha)$ is closely related to the more natural quantity $\w_n(\alpha)$, which is the number of non-zero digits among the first $n$ digits after the radix point in the (centred) base $b$ expansion of $\alpha$. For a careful definition of this, see Section \ref{sec6}. However, $\tilde \w_n$ has more convenient analytic properties.

Now we come to the main result of the section. As we said before, it is a little technical to state. However, it is rather less technical in the case $k = 2$, in which case the reader may wish to compare it with the outline in Section \ref{outline-sec}.

\begin{proposition}\label{sec3-main} Let $n \in \N$ be divisible by $k$ and set $N := b^n$. Suppose that $\delta \in (0,1]$ and that $|\widehat{\mu_n}(\theta)| \geq \delta$. Then there are $t_1,\dots, t_{k-1} \in \Z$ with $|t_j| \leq N$ for all $j$ and a positive integer $q_0 \leq b^{k^2}$ such that for at least $\frac{1}{2} \delta^{2^k} 2^{(k-1) n/k}$ choices of $\x^{(1)},\dots, \x^{(k-1)} \in \{0,1\}^{\zn{n/k}}$ we have
\[ \tilde\w_n \big( \theta q_0 \prod_{i=1}^{k-1}( L_{b^k}(\x^{(i)}) + t_i) \big) \leq 2^k b^{2k} \log(2/\delta). \]\end{proposition}
\begin{proof}
By \eqref{lb-form} and the definition of the measure $\mu_n$, we have
\begin{equation}\label{mun-hat-form}  \widehat{\mu_n}(\theta) = \E_{\x \in \{0,1\}^n} e \big(\theta \big(u + (d_2 - d_1) L_b(\x)\big)^k\big),\end{equation}
where $u := d_1(b^n - 1)/(b - 1)$. The first stage of the decoupling procedure is to split the variables $\x$ into $k$ disjoint subsets of size $n/k$. If $\x = (x_i)_{i \in \zn{n}} \in \{0,1\}^{\zn{n}}$, for each $j \in \zn{k}$ we write $\x^{(j)} = (x_{ik + j})_{i \in \zn{n/k}} \in \{0,1\}^{\zn{n/k}}$. Then
\begin{equation}\label{multisection} L_b(\x) = \sum_{j \in \zn{k}} b^j L_{b^k}(\x^{(j)}). \end{equation}
(Note here that $L_b$ is defined on $\{0,1\}^{\zn{n}}$, whereas $L_{b^k}$ is defined on $\{0,1\}^{\zn{n/k}}$.) By \eqref{mun-hat-form} we have
\[ \widehat{\mu_n}(\theta) = \E_{\x^{(0)},\dots, \x^{(k-1)} \in \{0,1\}^{\zn{n/k}}} e\bigg(\theta \big( u + (d_2 - d_1)\sum_{i \in \zn{k}} b^i L_{b^k}(\x^{(i)})\big)^k\bigg).\] 
Expanding out the $k$th power and collecting terms, this can be written as 
\[ \E_{\x^{(0)},\dots, \x^{(k-1)} \in \{0,1\}^{\zn{n/k}}} \big(\prod_{j \in \zn{k}} \Psi_j(\x) \big)e \big(\theta  q_0 \prod_{i \in \zn{k}} L_{b^k}(\x^{(i)})\big),\] where
\[ q_0 := k!(d_2 - d_1)^k b^{k(k-1)/2}\] and $\Psi_j$ is some $1$-bounded function of the variables $\x^{(i)}$, $i \in \zn{k} \setminus \{j\}$, the precise nature of which does not concern us. The inequality $q_0 \leq b^{k^2}$ follows using $|d_1 - d_1| \leq b$ and the estimate $k! \leq 3^{k(k-1)/2}$, since $b \geq 3$.

One may now apply the Cauchy-Schwarz inequality $k$ times to eliminate the functions $\Psi_j$ in turn. This procedure is well-known from the theory of hypergraph regularity \cite{gowers-hypergraph} or from the proofs of so-called generalised von Neumann theorems in additive combinatorics \cite{green-tao-linearprimes}. For a detailed statement, see Proposition \ref{box-gvn}. From this it follows that
\[ \delta^{2^k} \leq \E e \big(\theta q_0 \sum_{\omega \in \{0,1\}^{\zn{k}}}(-1)^{|\omega|} \prod_{i \in \zn{k}} L_{b^k}(\x_{\omega_i}^{(i)}  )\big),\]  where the average is over $\x^{(0)}_0, \dots \x_0^{(k-1)}, \x^{(0)}_1, \dots \x_1^{(k-1)} \in \{0,1\}^{[0, n/k)}$, and we write $\omega = (\omega_i)_{i \in [0,k)}$ and $|\omega| = \sum_{i = 1}^k |\omega_i|$. By pigeonhole there is some choice of $\x^{(0)}_1,\dots, \x^{(k-1)}_1$ such that the remaining average over $\x^{(0)}_0, \dots \x_0^{(k-1)}$ is at least $\delta^{2^k}$. This may be written as 
\[  \delta^{2^k} \leq \big| \E_{\x^{(0)},\dots, \x^{(k-1)} \in \{0,1\}^{\zn{n/k}}} e \big(\theta q_0 \prod_{i \in \zn{k}} (L_{b^k}(\x^{(i)}) + t_i) \big)\big| \]where $t_i := -L_{b^k}(\x_1^{(i)})$. It follows that for at least $\frac{1}{2}\delta^{2^k} 2^{(k-1)n/k}$ choices of $\x^{(1)},\dots, \x^{(k-1)} \in \{0,1\}^{\zn{n/k}}$ we have 
\begin{equation}\label{toret} \big| \E_{\x^{(0)}\in \{0,1\}^{[0, n/k)}} e \big(\theta q_0 L_{b^k}(\x^{(0)}) \prod_{i = 1}^{k-1} (L_{b^k}(\x^{(i)}) + t_i) \big)\big| \geq  \delta^{2^k}/2.\end{equation}

Let $\alpha \in \R/\Z$ be arbitrary. Note that 
\begin{align*} \tilde\w_n(\alpha)  = \sum_{i \in \zn{n-1}} & \Vert \alpha b^{i} \Vert^2   = \sum_{j \in \zn{k}} \sum_{i \in \zn{n/k}} \Vert \alpha b^{j + ik}\Vert^2  \\ & \leq \big(\sum_{j \in \zn{k}} b^{2j} \big)  \sum_{i \in \zn{n/k}} \Vert \alpha  b^{ik}\Vert^2 \leq b^{2k}  \sum_{i \in \zn{n/k}} \Vert \alpha  b^{ik}\Vert^2.\end{align*}
Therefore, using the inequality $|1 + e(t)| = 2|\cos(\pi t)| \leq 2 \exp(-\Vert t \Vert^2)$, we have
\begin{align*}
\big|\E_{\y \in \{0,1\}^{\zn{n/k}}} & e(\alpha L_{b^k}(\y))\big| = \prod_{i \in \zn{n/k}} \bigg| \frac{1 + e(\alpha b^{ik}) }{2} \bigg| \\ & \leq \exp\big( -\sum_{i \in [0, n/k)} \Vert \alpha b^{ik} \Vert^2 \big) \leq \exp \big(-b^{-2k} \tilde\w_n(\alpha)\big).
\end{align*}
Combining this with \eqref{toret}, Proposition \ref{sec3-main} follows.\end{proof}

\section{Sums of products of linear forms}\label{sec5}

We turn now to the next step of the outline in Section \ref{outline-sec}, which we called additive expansion. The main result of the previous section, Proposition \ref{sec3-main}, is roughly of the form ``for quite a few $m \sim N^{k-1}$, $\tilde\w_n(\theta m) \lesssim \log(2/\delta)$''. (The reader should not attach any precise meaning to the symbols $\sim, \lesssim$ here.) The shortcoming of the statement as it stands is that the set of $m$ is of size $\sim 2^{(k-1) n/k}$, which is substantially smaller than $N^{k-1}$ (recall that $N = b^n$). The aim of this section is to upgrade the conclusion of Proposition \ref{sec3-main} to get a much larger set of $m$. Here is the statement we will prove.

\begin{proposition}\label{sec4-main}
Set $C := b^{7k^2/2}$. Suppose that $\delta \in (0, 1]$ and that $k \mid n$. Suppose that $|\widehat{\mu_n}(\theta)| \geq \delta$, and that $N \geq (2/\delta)^{C}$, where $N := b^n$. Then for at least $(\delta/2)^{C}N^{k-1}$ values of $m$, $ |m| \leq C N^{k-1}$, we have $\tilde\w_n(\theta m) \leq C\log(2/\delta)$.
\end{proposition}

The basic idea of the proof is to take sums of a few copies of the set of $m$ produced in Proposition \ref{sec3-main} which (it turns out) expands this set of $m$ dramatically, whilst retaining the property of $\tilde\w_n(\theta m)$ being small.

We assemble some ingredients. The key input is Theorem \ref{bkmp-thm} (see, in addition to Section \ref{outline-sec}, Appendix \ref{appB}). We will also require some other lemmas of a miscellaneous type, and we turn to these first.

\begin{lemma}\label{div-lem}
Let $\eps, U, V$ be real parameters with $0 < \eps \leq 2^{-44}$ and $U, V \geq 64/\eps$. Suppose that $\Omega \subset [-U, U] \times [-V, V]$ has size at least $\eps UV$. Then at least $\eps^7 UV$ integers $n \in [-2UV, 2UV]$ may be written as $u_1 v_1 + u_2v_2$ with $(u_1, v_1), (u_2, v_2) \in \Omega$.
\end{lemma}
\begin{proof} The conclusion is invariant under applying any of the four involutions $(u,v) \mapsto (\pm u, \pm v)$ to $\Omega$, so without loss of generality we may suppose that $\Omega \cap ([0, U] \times [0, V])$ has size at least $\eps UV/4$. It then follows that $\Omega \cap ([\eps U/32, U] \times [\eps V/32, V])$ has size at least $\eps UV/8$. Covering this box by disjoint dyadic boxes $[2^i, 2^{i+1})\times [2^j, 2^{j+1})$ contained in $[\eps U/64, 2U] \times [\eps V/64, 2V]$, we see that there is some dyadic box $[U', 2U') \times [V', 2V')$, $\eps U/64 \leq U' \leq U$, $\eps V/64 \leq V' \leq V$, on which the density of $\Omega$ is at least $\eps/32$. Without loss of generality, suppose that $U' \leq V'$, and set $X := U' V' \geq 1$.  Set $\Omega':= \Omega \cap \big([U', 2U') \times [V', 2V')\big)$.

For $n \in \Z$, denote by $r(n)$ the number of representations of $n$ as $u_1 v_1 + u_2 v_2$ with $(u_1, v_1), (u_2, v_2) \in \Omega'$, and by $\tilde r(n)$ the number of representations as $u_1 v_1 + u_2 v_2$ with $(u_1, v_1), (u_2, v_2) \in [U', 2U') \times [V', 2V')$. Thus $r(n) \leq \tilde r(n)$. By Cauchy-Schwarz, 
\begin{equation}\label{comp-1} (\eps X/32)^4  \leq |\Omega'|^4 = \big(\sum_n r(n)\big)^2  \leq |\Supp(r)| \sum_n \tilde r(n)^2.\end{equation}
Now, denoting by $\nu(n)$ the number of divisors of $n$ in the range $[U', 2U')$, \[ \tilde r(n)  \leq \sum_{m \leq 4X}  \nu(m) \nu(n - m)  = \sum_{\substack{d, e \in [U', 2U') \\ (d,e) | n}} \sum_{\substack{m \leq 4X \\  d \mid m, e \mid n - m}}  1 \leq 8X \!\!\!\sum_{\substack{d,e \in [U', 2U') \\ (d,e) | n}} \frac{1}{[d,e]} .\] 
Here, in the last step we used the fact that the set of $m$ satisfying $d \mid m$ and $e \mid n - m$ is a single residue class modulo $[d,e]$ (the lowest common multiple of $d$ and $e$), whose intersection with the interval $[1, 4X)$ has size $\leq 1 + 4X /[d, e] \leq 8 X/[d,e]$ since $[d, e] \leq (2U')^2 \leq 4X$. 

Setting $\delta := (d,e)$ and $d = \delta d'$, $e = \delta e'$, so that $[d,e] = \delta d' e'$, it then follows that 
\[ \tilde r(n) \leq 8 X \sum_{ \delta \mid n} \frac{1}{\delta} \sum_{d', e' \in [U'/\delta, 2U'/\delta)} \frac{1}{d' e'} \leq 8 X \sum_{\delta | n} \frac{1}{\delta}.\]
Since $\tilde r(n)$ is supported where $n \leq 8X$, we have
\begin{align*} \sum_{n} \tilde r(n)^2 & \leq  (8X)^2 \sum_{n \leq 8X} \big(\sum_{\delta | n} \frac{1}{\delta}\big)^2 = (8X)^2 \sum_{\delta_1, \delta_2 \leq 8X} \frac{1}{\delta_1 \delta_2} \sum_{n \leq 8X} 1_{[\delta_1, \delta_2] | n} \\ & \leq (8X)^2 \sum_{\delta_1, \delta_2 \leq 8X} \frac{1}{\delta_1 \delta_2} \big(\frac{8X}{[\delta_1, \delta_2]} + 1\big) .\end{align*}
The contribution from the $+1$ term is $\leq (8X)^2 (1 + \log 8X)^2 < 2^{10} X^3$, since $X \geq 1$. Since $[\delta_1, \delta_2] \geq \sqrt{\delta_1 \delta_2}$, the contribution from the main term is $\leq 2^8 \zeta(\frac{3}{2})^2 X^3 < 2^{11}X^3$. It follows that $\sum_n \tilde r(n)^2 \leq 2^{12} X^3$.
Comparing with \eqref{comp-1}, we obtain $|\Supp(r)| \geq 2^{-32} \eps^4 X \geq 2^{-44} \eps^6 UV$. Since we are assuming that $\eps \leq 2^{-44}$, this is at least $\eps^7 UV$, and the proof is complete.
\end{proof}

\begin{lemma}\label{inter-diffs}
Let $X \geq 1$ be real, and suppose that $S_1,\dots, S_t \subseteq [-X, X]$ are sets of integers with $|S_i| \geq \eta X$. Then $\big| \bigcap_{i=1}^t (S_i - S_i) \big|\geq (\eta/5)^t X$.
\end{lemma}
\begin{proof}
We have \[ \sum_{h_2,\dots, h_t} \big( \sum_x 1_{S_1}(x) 1_{S_2} (x+h_2) \cdots1_{S_t} (x + h_t) \big) = \prod_{i = 1}^t |S_i|  \geq \eta^t X^t.\] Since the $h_i$ may be restricted to range over $[-2X, 2X]$, which contains at most $5X$ integers, there is some choice of $h_2,\dots, h_t$ so that $\sum_x 1_{S_1}(x) 1_{S_2} (x+h_2) \cdots1_{S_t} (x + h_t)  \geq (\eta/5)^t X$.
That is, there is a set $S$, $|S| \geq (\eta/5)^t X$, such that $S \subseteq S_1 \cap (S_2 - h_2) \cap \cdots \cap (S_t - h_t)$. But then $S - S \subseteq \bigcap_{i=1}^t (S_i - S_i)$, and the result is proved.
\end{proof}

We turn now to the heart of proof of Proposition \ref{sec4-main}. The key technical ingredient is the following.

\begin{proposition}\label{gk-r}
Let $d, r$ be positive integers with $d \geq 2$. Let $\alpha \in (0,1]$. Let $m$ be an integer, set $N := d^m$, and suppose that $N \geq (2/\alpha)^{(32d)^r}$. Suppose that $t_1,\dots, t_r$ are integers with $|t_j| \leq N$.  Define $L_d: \{0,1\}^{\zn{m}} \rightarrow \zn{N}$ as in \eqref{L-def}. Suppose that $A \subset \big( \{0,1\}^{\zn{m}} \big)^{r}$ is a set of size at least $\alpha 2^{mr}$. Then at least $(\alpha/2)^{(32d)^r}N^r$ integers $x$ with $|x| \leq (8d N)^r$ may be written as a $\pm$ sum of at most $(4d)^r$ numbers $\prod_{j=1}^r (L_d(\y_j) + t_j)$ with $(\y_1,\dots, \y_{r}) \in A$.
\end{proposition}
\begin{proof} It is convenient to write $\ltrans_j(\y) := L_d(\y) + t_j$, $j = 1,\dots, r$. Note for further use the containment
\begin{equation}\label{crude-cont} \ltrans_j(\{0,1\}^{\zn{m}}) \subset [-2N, 2N],\end{equation} which follows from the fact that $|t_j| \leq N$.

Turning to the proof, we proceed by induction on $r$. In the case $r = 1$, we can apply Theorem \ref{bkmp-thm}. Noting that $L_d(\{0,1,\dots, d-1\}^m) = \{0,1,\dots, N - 1\}$, we see that at least $\alpha^{\log_2 d} N$ elements of $\{0,1,\dots, N-1\}$ are the sum of $d-1$ elements $L_d(\y_1)$, $\y_1 \in A$. Since, for any $\y_1^{(1)},\dots, \y_1^{(d-1)} \in A$, we have
\[ \sum_{i = 1}^{d-1} \ltrans_1(\y_1^{(i)}) = \sum_{i = 1}^{d-1} L_d(\y_1^{(i)}) + (d - 1) t_1,\] we see that at least $\alpha^{\log_2 d} N$ elements of $[-d N, d N]$ are the sum of $d - 1$ elements $\ltrans_1(\y_1)$, $\y_1 \in A$, which gives the required result in this case.

Now suppose that $r \geq 2$, and that we have proven the result for smaller values of $r$.  For each $\y_r \in \{0,1\}^{[0,m)}$, denote by $A(\y_r) \subseteq (\{0,1\}^{[0,m)})^{r-1}$ the maximal set such that $A(\y_r) \times \{\y_r\} \subseteq A$. By a simple averaging argument there is a set $Y$ of least $(\alpha/2) 2^{m}$ values of $\y_r$ such that $|A(\y_r)| \geq (\alpha/2) 2^{m(r-1)}$. By the inductive hypothesis, for each $\y_r \in Y$ there is a set \begin{equation}\label{b-fibr-cont} B(\y_r) \subseteq [- (8d N)^{r-1}, (8d N)^{r-1}],\end{equation} with \begin{equation}\label{b-fibr-size} |B(\y_r)| \geq (\alpha/4)^{(32d)^{r-1}} N^{r-1},\end{equation} such that everything in $B(\y_r)$ is a $\pm$ combination of at most $(4d)^{r-1}$ elements $\ltrans_1(\y_1) \cdots \ltrans_{r-1}(\y_{r-1})$ with $(\y_1,\dots, \y_{r-1}) \in A(\y_r)$. Observe that everything in $(B(\y_r) - B(\y_r)) \ltrans_r(\y_r)$ is then a $\pm$ combination of at most $2 (4d)^{r-1}$ elements $\ltrans_1(\y_1) \cdots \ltrans_{r}(\y_r)$ with $(\y_1,\dots, \y_r) \in A$.

Suppose now that $z \in (d - 1) \ltrans_r(Y) = \ltrans_r(Y) + \cdots + \ltrans_r(Y)$. Note that, by \eqref{crude-cont}, \begin{equation}\label{z-bd-crude} |z| < 2dN.\end{equation} For each such $z$, pick a representation $z = \ltrans_r(\y_r^{(1)}) + \cdots + \ltrans_r(\y_r^{(d-1)})$ with $\y_r^{(i)} \in Y$ for $i = 1,\dots, d-1$, and define $S(z) := \bigcap_{i=1}^{d-1} (B(\y_r^{(i)}) - B(\y_r^{(i)}))$. By \eqref{b-fibr-cont}, \eqref{b-fibr-size} and Lemma \ref{inter-diffs} (taking $X := (8dN)^{r-1}$, $\eta := (8d)^{-(r-1)} (\alpha/4)^{(32d)^{r-1}}$ and $t := d - 1$ in that lemma) we have
\begin{align}\nonumber |S(z)| & \geq 5^{-(d-1)} (8d)^{-(r-1)(d-2)} (\alpha/4)^{(32 d)^{r-1} (d - 1)} N^{r-1} \\ & \geq (\alpha/2)^{4d (32d)^{r-1}} N^{r-1}.\label{sm-lower} \end{align}
Here, the second bound is crude and uses the inequality \[ (2d+2)(32d)^{r-1}  \geq (d-1)\log_2 5 + (r-1)(d-2) \log_2(8d),\] valid for $d \geq 2$ and $r \geq 1$ (by a large margin if $r > 1$).

Note that everything in $S(z) z$ is a $\pm$ combination of at most $2(d-1) (4d)^{r-1}$ elements $\ltrans_1(\y_1) \cdots \ltrans_r( \y_r)$ with $(\y_1,\dots, \y_r) \in A$. Set $\Omega := \bigcup_{z \in (d-1) \ltrans_r(Y)} (S(z) \times \{z\})$. Then $\Omega \subset [-U, U] \times [-V, V]$ where by \eqref{b-fibr-cont} and \eqref{z-bd-crude} we can take
$U := 2(8dN)^{r-1}$ and $V := 2 dN$. Now by Theorem \ref{bkmp-thm}, and recalling that $|Y| \geq (\alpha/2) 2^m$, we have $| (d - 1) \ltrans_r(Y) | = |(d - 1) L_d(Y) | \geq (\alpha/2)^{\log_2 d} N$. From this and \eqref{sm-lower}, we have $|\Omega| \geq (\alpha/2)^{4d(32d)^{r-1} + \log_2 d} N^r$. Thus, noting that $UV = 2^{3r -1 + r \log_2 d} N^r$, it follows that $|\Omega| \geq \eps UV$ with 
\begin{equation}\label{eps-def} \eps := (\alpha/2)^{4d(32d)^{r-1} + 3r + (r+1) \log_2 d}.\end{equation}
Now we aim to apply Lemma \ref{div-lem}. For such an application to be valid, we require $\eps < 2^{-44}$, which is comfortably a consequence of \eqref{eps-def}. We also need that $U,V \geq 64/\eps$, which follows from \eqref{eps-def} and the lower bound on $N$ in the hypotheses of the proposition. Note that if $(u_1, v_1) = (S(z), z)$, $(u_2, v_2) = (S(z'), z') \in \Omega$ then $u_1 v_1 + u_2 v_2 = S(z) z' + S(z') z'$ is a $\pm$ combination of at most $(4d)^r$ elements $\ltrans_1(\y_1) \cdots \ltrans_r( \y_r)$ with $(\y_1,\dots, \y_r) \in A$, and by Lemma \ref{div-lem} there are $\geq \eps^7 UV > \eps^7 N^r$ such elements. To conclude the argument, we need only check that $\eps^7 \geq (\alpha/2)^{(32 d)^r}$ which, using \eqref{eps-def}, comes down to checking that $4d (32 d)^{r-1} \geq 7(3r + (r + 1) \log_2 d)$, which is comfortably true for all $d, r \geq 2$.\end{proof}

Finally, we are ready for the proof of the main result of the section, Proposition \ref{sec4-main}, which results from combining Propositions \ref{sec3-main} and \ref{gk-r}.

\begin{proof}[Proof of Proposition \ref{sec4-main}] In the following proof we suppress a number of short calculations, showing that various constants are bounded by $C = b^{7k^2/2}$. These calculations are all simple finger exercises using the assumption that $b \geq 3$ and $k \geq 2$. 

First apply Proposition \ref{sec3-main}. As in the statement of that Proposition, we obtain $t_1,\dots, t_{k-1} \in \Z$, $|t_j| \leq N$ such that, for at least $\frac{1}{2} \delta^{2^k} 2^{(k-1) n/k}$ choices of $\x^{(1)},\dots, \x^{(k-1)} \in \{0,1\}^{\zn{n/k}}$ we have
\begin{equation}\label{3-main-w} \tilde \w_n \big( \theta q_0 \prod_{i=1}^{k-1} (L_{b^k}(\x^{(i)}) + t_i)  \big) \leq 2^k b^{2k} \log(2/\delta),\end{equation} for some positive integer $q_0 \leq b^{k^2}$.
(For the definition of $\tilde\w_n$, see Definition \ref{wn-tilde-def}.) To this conclusion, we apply Proposition \ref{gk-r}, taking $m := n/k$, $r := k-1$ and $d := b^k$ in that proposition, and taking $A$ to be the set of all $(\x^{(1)},\dots, \x^{(k-1)})$ as just described; thus we may take $\alpha := \delta^{2^k}/2$. Note that $N = d^m = b^n$ is the same quantity. The reader may check that the lower bound on $N$ required for this application of Proposition \ref{gk-r} is a consequence of the assumption on $N$ in Proposition \ref{sec4-main}.

We conclude that at least $(\delta^{2^k}/4)^{(32 b^k)^{k-1}} N^{k-1} > (\delta/2)^{C} N^{k-1}$ integers $x$ with $|x| \leq (8b^k N)^{k-1}$ may be written as a $\pm$ sum of at most $(4b^k)^{k-1}$ numbers $\prod_{i=1}^{k-1} (L_{b^k}(\x^{(i)}) + t_i) $, with $(\x^{(1)},\dots, \x^{(k-1)}) \in A$. By \eqref{3-main-w}, the fact that $\tilde\w_n(-\alpha) = \tilde\w_n(\alpha)$, as well as the (easily-verified) subadditivity property
\[ \tilde\w_n(\alpha_1 + \cdots + \alpha_s) \leq s (\tilde\w_n(\alpha_1) + \dots + \tilde\w_n(\alpha_s)),\] we see that for all such $x$ we have
\[ \tilde\w_n(\theta q_0 x) \leq (4b^k)^{2(k-1)} 2^k b^{2k} \log(2/\delta) <  C \log(2/\delta).\] 
Finally, note that for all these $x$ we have $|q_0 x| \leq b^{k^2}(8 b^k)^{k-1} N^{k-1}$, which is less than $C N^{k-1}$. This concludes the proof.
\end{proof}

\section{From digital to diophantine}\label{sec6}

In this section we turn to the final step in the outline of Section \ref{outline-sec}, the aim of which is to convert the `digital' conclusion of Proposition \ref{sec4-main} to  the `diophantine' conclusion of Proposition \ref{weyl-restricted}. Before turning to detailed statements, we comment on the notion of a centred base $b$ expansion.\vspace*{8pt}

\emph{Centred base $b$ expansions.} Consider $\alpha \in \R/\Z$. Then there are essentially unique choices of integers $\alpha_j \in (-\frac{b}{2}, \frac{b}{2}]$ such that \begin{equation}\label{centred-exp}  \alpha = \alpha_0 + \alpha_1 b^{-1} + \alpha_2 b^{-2} + \dots \md{1}.\end{equation} We call this the \emph{centred} base $b$ expansion of $\alpha \md{1}$.

Let us pause to explain the existence of such expansions. When $b$ is odd, so that $(-\frac{b}{2}, \frac{b}{2}] = \{-\frac{1}{2}(b-1),\dots, \frac{1}{2}(b-1)\}$, the centred expansion may be obtained from the more usual base $b$ expansion of $\alpha + \frac{b}{2}$, noting that $\frac{b}{2} = \frac{1}{2}(b-1) (1 + b^{-1} + b^{-2} + \cdots)$. As usual, there is some ambiguity when all the digits from some point on are $\frac{1}{2}(b-1)$; any such number can also be written with all digits from some point on being $-\frac{1}{2}(b-1)$. For consistency with the usual base $b$ expansions, we always prefer the latter representation. When $b$ is even, so that $(-\frac{b}{2}, \frac{b}{2}] = \{-\frac{1}{2}(b-2), \dots, \frac{1}{2}b\}$, one instead considers the usual base $b$ expansion of $\alpha + \frac{b(b-2)}{2(b-1)}$, noting now that $\frac{b(b-2)}{2(b-1)} = \frac{1}{2}(b-2) (1 + b^{-1} + b^{-2} + \cdots)$.

\begin{definition}\label{wn-def}  Given $\alpha \in \R/\Z$, denote by $\w_n(\alpha)$ the number of nonzero digits among the first $n$ digits $\alpha_0,\alpha_1,\dots, \alpha_{n-1}$ in the centred expansion \eqref{centred-exp}.\end{definition}

We record the connection between $\w_n$ and the `analytic' proxy $\tilde\w_n$, introduced in Definition \ref{wn-tilde-def}.

\begin{lemma}\label{ww-tilde} Suppose that $b \geq 3$. Then $\tilde \w_n(\alpha) \leq \w_n(\alpha) \leq 16b^2 \tilde \w_n(\alpha)$.
\end{lemma}
\begin{proof} Let the centred expansion of $\alpha \md{1}$ be \eqref{centred-exp}, and suppose that $\alpha_i$ is a non-zero digit. We have $\alpha b^{i-1} \equiv \sum_{j \geq 0} \alpha_{i +j} b^{-j - 1} \md{1}$. However, 
\[  |\sum_{j \geq 0} \alpha_{i +j} b^{-j - 1} | \leq \frac{b}{2} \sum_{j \geq 0} b^{-j-1} = \frac{b}{2(b-1)} \leq \frac{3}{4},\] and, since $\alpha_i \neq 0$,
\[  |\sum_{j \geq 0} \alpha_{i +j} b^{-j - 1} |  \geq \frac{1}{b} - \frac{b}{2}\sum_{j \geq 1} b^{-j-1} = \frac{b-2}{2b(b-1)} \geq \frac{1}{4b}.\]
Thus $\Vert \alpha b^{i-1} \Vert \geq 1/4b$, and the upper bound follows.

The lower bound is not needed elsewhere in the paper, but we sketch the proof for completeness. Let $I := \{ i : \alpha_i \neq 0\}$. Given $j$, denote by $i(j)$ the distance from $j$ to the smallest element of $I$ which is greater than $j$. Then 
\[ \Vert \alpha b^j\Vert = \Vert \sum_{i \in I, i > j} \alpha_i b^{-i + j}\Vert \leq \frac{b}{2} \sum_{m \geq i(j)} b^{-m} = \frac{b^2}{2(b-1)} b^{-i(j)} .\]
Now square this and sum over $j$, and use the fact that $\# \{ j : i(j) = i\} \leq |I| = \w_n(\alpha)$ for all $i$.
\end{proof}
\emph{Remarks.} This upper bound breaks down when $b = 2$, as may be seen by considering $\alpha$ of the form $1 - 2^{-m}$. This is the main reason for the restriction to $b \geq 3$ in the paper. \vspace*{8pt}

Here is the main result of the section.  

\begin{proposition}\label{sec-5-main} Let $b \geq 3$ be an integer. Let $r, M, n$ be positive integers, and set $N := b^n$. Let $\eta \in (0,1]$ be real. Suppose that $M, N \geq b^{20r} \eta^{-2}$. Suppose that $\theta \in \R$, and that $\w_n(\theta m) \leq r$ for at least $\eta M$ values of $m \in [-M, M]$.  Then there is some positive integer $q \leq b^{20r} \eta^{-2}$ such that $\Vert \theta q \Vert \leq b^{20r} \eta^{-2} M^{-1} N^{-1}$.
\end{proposition}

Before giving the proof, we assemble some lemmas. In the first of these, we will again be concerned with centred expansions in base $b$, but this time of integers.  Every integer $x$ has a unique finite-length centred base $b$ expansion 
\begin{equation}\label{centred-exp-int} x = x_0 + x_1 b + x_2 b^2 + \dots.\end{equation} with $x_i \in (-\frac{b}{2}, \frac{b}{2}]$. To see uniqueness, note that $x_0$ is uniquely determined by $x \md{b}$, then $x_1$ is uniquely determined by $\frac{x - x_0}{b} \md{b}$, and so on. Strictly speaking, we do not need the existence in this paper but one way to see it is to take the usual base $b$ expansion and modify from the right. For instance, in base 10 we have, denoting the `digit' $-d$ by $\overline{d}$, $6277 = 628\overline{3} = 63\overline{2}\overline{3} = 1\overline{4}3\overline{2}\overline{3}$.

Denote by $\dd_b(x)$ the number of nonzero digits in this expansion of $x$. The set of $x$ for which $\dd_b(x) \leq r$ is a kind of ``digital Hamming ball''. As for true Hamming balls \cite{bonami, ks} subsets of this set have little additive structure. Such a result was stated as Proposition \ref{add-digit-hamming}. We recall the statement now. Recall that, if $A \subset \Z$ is a finite set, the additive energy $E(A)$ is the number of quadruples $(a_1, a_2, a_3, a_4) \in A \times A \times A \times A$ with $a_1 + a_2 = a_3 + a_4$.

\begin{add-digit-hamming-rpt} 
Let $r \in \Z_{\geq 0}$. Suppose that $A \subset \Z$ is a finite set, all of whose elements have at most $r$ nonzero digits in their centred base $b$ expansion. Then $E(A) \leq (2b)^{4r} |A|^{2}$.
\end{add-digit-hamming-rpt}
The proof of Proposition \ref{add-digit-hamming} will proceed by induction. However, to make this work, we need to prove a more general statement, involving four potentially different sets $A_1, A_2, A_3, A_4$ instead of just one, as well as the provision for a `carry' in base $b$ arithmetic. Here is the more general statement, from which Proposition \ref{add-digit-hamming} follows immediately.

\begin{lemma}\label{lem63}
Let $r_1,r_2, r_3, r_4 \in \Z_{\geq 0}$. For each $i \in \{1,2,3,4\}$, suppose that $A_i \subset \Z$ is a finite set, all of whose elements have at most $r_i$ nonzero digits in their centred base $b$ expansion. Let $e \in \Z$, $|e| < b$. Then the number of quadruples $(a_1, a_2, a_3, a_4) \in A_1 \times A_2 \times A_3 \times A_4$ with $a_1 + a_2 = a_3 + a_4 + e$ is at most $(2b)^{r_1 + r_2 + r_3 + r_4} |A_1|^{1/2} |A_2|^{1/2} |A_3|^{1/2} |A_4|^{1/2}$.
\end{lemma}
\begin{proof} We proceed by induction on $\sum_{j = 1}^4 |A_j| + \sum_{j = 1}^4 r_j$, the result being obvious when this quantity is zero. Suppose now that $\sum_{j = 1}^4 |A_j| + \sum_{j = 1}^4 r_j = n > 0$ and that the result has been proven for all smaller values of $n$. If any of the $A_j$ are empty, or if $A_1 = A_2 = A_3 = A_4 = \{0\}$, the result is obvious. 

Suppose this is not the case, but that $b$ divides every element of $\bigcup_{j = 1}^4 A_j$. Let $b^m$ be the largest power of $b$ which divides every element of $\bigcup_{j = 1}^4 A_j$, this being well-defined since this set contains at least one nonzero element. Then, if the number of quadruples in $A_1 \times A_2 \times A_3 \times A_4$ with $a_1 + a_2 = a_3 + a_4 + e$ is nonzero, we must have $e = 0$, and the number of such quadruples is the same as the number in $\frac{1}{b^m}A_1 \times \frac{1}{b^m} A_2 \times \frac{1}{b^m} A_3 \times \frac{1}{b^m} A_4$. Thus, replacing $A_j$ by $\frac{1}{b^m} A_j$, we may assume that not all the elements of $\bigcup_{j=1}^4 A_j$ are divisible by $b$.

For each $j \in \{1,2,3,4\}$ and for each $i \in (-\frac{b}{2}, \frac{b}{2}]$, write $A_j^{(i)}$ for the set of $x \in A_j$ whose first digit $x_0$ (in the centred base $b$ expansion \eqref{centred-exp-int}) is $i$. Write $\alpha_j(i)$ for the relative density of $A_j^{(i)}$ in $A_j$, that is to say $|A_j^{(i)}| = \alpha_j(i) |A_j|$. Any quadruple $(a_1, a_2, a_3, a_4)$ with $a_1 + a_2 = a_3 + a_4 + e$ must have $a_j \in A_j^{(i_j)}$, where $i_1 + i_2 \equiv i_3 + i_4 + e \md{b}$. Let us estimate the number of such quadruples $(a_1, a_2, a_3, a_4)$, for each quadruple $(i_1, i_2, i_3, i_4) \in (-\frac{b}{2}, \frac{b}{2}]^4$ satisfying this condition.

First note that $i_1 + i_2  = i_3 + i_4 + e + e' b$ for some integer $e'$, where
\[ |e'| \leq \frac{1}{b} \big( |i_1 + i_2 - i_3 - i_4| + |e|\big) \leq \frac{3(b-1)}{b} < b,\] where here we noted that $|i_1 - i_3|, |i_2 - i_4|, |e| \leq b - 1$. We then have $\frac{1}{b}(a_1 - i_1) + \frac{1}{b}(a_2 - i_2) - \frac{1}{b}(a_3 - i_3) - \frac{1}{b}(a_4 - i_4) = - e'$.
Now the set $A'_j := \frac{1}{b}(A_j^{(i_j)} - i_j)$ is a finite set of integers, all of whose elements $x$ have $\dd_b(x) \leq r'_j := r_j - 1_{i_j \neq 0}$. Note that $\sum_{j = 1}^4 |A'_j| + \sum_{j = 1}^4 r'_j < \sum_{j = 1}^4 |A_j| + \sum_{j = 1}^4 r_j$; if any $i_j$ is not zero, this follows from the fact that $r'_j = r_j - 1$, whereas if $i_1 = i_2 = i_3 = i_4 = 0$ we have $\sum_{j = 1}^4 |A'_j| = \sum_{j = 1}^4 |A_j^{(0)}| < \sum_{j = 1}^4 |A_j|$, since not every element of $\bigcup_{j = 1}^4 A_j$ is a multiple of $b$.

It follows from the inductive hypothesis that the numbers of quadruples $(a_1, a_2, a_3, a_4)$ with $a_1 + a_2 = a_3 + a_4 + e$, and with $a_j \in A_j^{(i_j)}$, $j = 1,\dots, 4$, is bounded above by $(2b)^{r_1 + r_2 + r_3 + r_4  - \#\{j : i_j \neq 0\}} \prod_{j = 1}^4 |A_j^{(i_j)}|^{1/2}$. To complete the inductive step, it is therefore enough to show that 
\begin{equation}\label{enough-additive} \sum_{i_1 + i_2 \equiv i_3 + i_4 + e \md{b}} (2b)^{-\# \{j : i_j \neq 0\}} \prod_{j=1}^4 \alpha_j(i_j)^{1/2} \leq 1.\end{equation}
If $e \not\equiv 0 \md{b}$ then we have $\# \{j : i_j \neq 0\} \geq 1$ for all $(i_1, i_2, i_3, i_4)$ in this sum, and moreover (where all congruences are $\md{b}$)

\begin{align*}
& \sum_{i_1 + i_2 \equiv i_3 + i_4 + e}  \prod_{j=1}^4 \alpha_j(i_j)^{1/2} \\ & = \sum_{x \in \Z/b\Z} \big (\sum_{i_1 + i_2 \equiv x+e } \alpha_1(i_1)^{1/2} \alpha_2(i_2)^{1/2} \big) \big( \sum_{i_3 + i_4 \equiv x } \alpha_3(i_3)^{1/2} \alpha_4(i_4)^{1/2}\big)  \\ & \leq \sum_{x \in \Z/b\Z} \big(\sum_{i_1 + i_2 \equiv x+e } \frac{\alpha_1(i_1) +  \alpha_2(i_2)}{2} \big) \big(\sum_{i_3 + i_4 \equiv x } \frac{\alpha_3(i_3) +  \alpha_4(i_4)}{2}\big) = b,\end{align*} since $\sum_i \alpha_j(i) = 1$ for each $j$. Therefore \eqref{enough-additive} holds in this case.

Suppose, then, that $e \equiv 0\md{b}$, which means that $e = 0$.
Then, if $i_1 + i_2 \equiv i_3 + i_4 \md{b}$ we either have $(i_1, i_2, i_3, i_4) = (0,0,0,0)$, or else $\# \{j : i_j \neq 0\} \geq 2$, and so to establish \eqref{enough-additive} it suffices to show
\begin{equation}\label{main-suffice} \prod_{j = 1}^4 \alpha_j(0)^{1/2} + (2b)^{-2} \sum_{\substack{i_1 + i_2 \equiv i_3 + i_4 \md{b} \\ (i_1, i_2, i_3, i_4) \neq (0,0,0,0)}}  \prod_{j=1}^4 \alpha_j(i_j)^{1/2} \leq 1.\end{equation}
Write $\eps_j := 1 - \alpha_j(0)$. We first estimate the contribution to the sum where none of $i_1, i_2, i_3, i_4$ is zero. We have, similarly to the above (and again with congruences being $\md{b}$)
\begin{align*}
&\!\!\!\! \sum_{\substack{i_1 + i_2 \equiv i_3 + i_4 \\ i_1i_2i_3i_4 \neq 0}}   \prod_{j=1}^4 \alpha_j(i_j)^{1/2} \\ & = \sum_{x \in \Z/b\Z} \big (\sum_{\substack{i_1 + i_2 \equiv x \\ i_1i_2 \neq 0}} \alpha_1(i_1)^{1/2} \alpha_2(i_2)^{1/2} \big) \big( \sum_{\substack{i_3 + i_4 \equiv x \\ i_3i_4 \neq 0}} \alpha_3(i_3)^{1/2} \alpha_4(i_4)^{1/2}\big)  \\ & \leq \sum_{x \in \Z/b\Z} \sum_{\substack{i_1 + i_2 \equiv x\\ i_1i_2 \neq 0}} \big(\frac{\alpha_1(i_1) +  \alpha_2(i_2)}{2} \big)\big(\sum_{\substack{i_3 + i_4 \equiv x\\i_3i_4 \neq 0}} \frac{\alpha_3(i_3) +  \alpha_4(i_4)}{2}\big) \\ & \leq b \big(\frac{\eps_1 + \eps_2}{2}\big)\big(\frac{\eps_3 + \eps_4}{2}\big) < b\sum_{j = 1}^4 \eps_j.\end{align*}
Next we estimate the contribution to the sum in \eqref{main-suffice} from the terms where at least one, but not all, of $i_1, i_2, i_3, i_4$ are zero. In each such term, at least two $i_j, i_{j'}$ are not zero, say with $j < j'$. Fix a choice of $j,j'$. Then for each $i_j, i_{j'}$ there are at most two choices of the other $i_{t}$, $t \in \{1,2,3,4\} \setminus \{j,j'\}$, one of which must be zero, and the other then being determined by the relation $i_1 + i_2 \equiv i_3 + i_4 \md{b}$.  It follows that the contribution to the sum in \eqref{main-suffice} from this choice of $j,j'$ is 
\begin{align*} \leq 2  \sum_{i_j, i_{j'} \neq 0} \alpha_j(i_j)^{1/2} \alpha_{j'}(i_{j'})^{1/2} & = 2 \big( \sum_{i \neq 0} \alpha_j(i)^{1/2} \big) \big( \sum_{i \neq 0} \alpha_{j'}(i)^{1/2}\big) \\ & \leq 2b \eps_j^{1/2}\eps_{j'}^{1/2} \leq b(\eps_j + \eps_{j'}) ,\end{align*} where in the middle step we used Cauchy-Schwarz and the fact that $\sum_{i \neq 0} \alpha_j(i) = \eps_j$. Summing over the six choices of $j,j'$ gives an upper bound of $3b \sum_{j = 1}^4 \eps_j$. Putting all this together, we see that the LHS of \eqref{main-suffice} is bounded above by $\prod_{j = 1}^4 (1 - \eps_j)^{1/2} + \frac{1}{b} \sum_{j = 1}^4 \eps_j$. Using $\prod_{j = 1}^4 (1 - \eps_j)^{1/2} \leq 1 - \frac{1}{2}\sum_{j = 1}^4 \eps_j$, it follows that this is at most $1$. This completes the proof of \eqref{main-suffice}, and hence of Lemma \ref{lem63}.
\end{proof}

Now we turn to the proof of Proposition \ref{sec-5-main}. 

\begin{proof}[Proof of Proposition \ref{sec-5-main}] Consider the map $\psi : \R \rightarrow \Z$ defined as follows. If $\alpha \md{1}$ has centred base $b$ expansion as in \eqref{centred-exp}, set $\psi(\alpha) := \alpha_0 b^{n-1} + \dots + \alpha_{n-2} b + \alpha_{n-1}$.
Observe that 
\begin{equation}\label{w-d} \dd_b(\psi(\alpha)) = \w_n(\alpha).\end{equation}
Note that 
\begin{equation}\label{pi-basic-prop} \Vert \alpha - b^{1-n} \psi(\alpha) \Vert \leq \sum_{i \geq n} \frac{b}{2} b^{-i} \leq \frac{3}{4} b^{1-n}.\end{equation} Thus if $\alpha_1 + \alpha_2 = \alpha_3 + \alpha_4$ then 
\[ \Vert b^{1-n} (\psi(\alpha_1) + \psi(\alpha_2) - \psi(\alpha_3) - \psi(\alpha_4)) \Vert \leq 3 b^{1-n}.\] Note also that, since $\psi$ takes values in $\Z \cap [-\frac{3}{4} b^n, \frac{3}{4} b^n]$, we have 
\[ |\psi(\alpha_1) + \psi(\alpha_2) - \psi(\alpha_3) - \psi(\alpha_4)| \leq 3 b^n.\]
Now if $x \in \Z$ is an integer with $\Vert b^{1-n} x \Vert \leq 3 b^{1-n}$ and $|x| \leq 3 b^n$ then $x$ takes (at most) one of the $7(6b+1)$ values $\lambda b^{n-1} + \lambda'$, $\lambda \in \{-3b, \dots, 3b\}$, $\lambda' \in \{0, \pm 1, \pm 2, \pm 3\}$. Denoting by $\Sigma$ the set consisting of these $7(6b + 1)$ values, we see that $\psi$ has the following almost-homomorphism property: if $\alpha_1 + \alpha_2 = \alpha_3 + \alpha_4$ then 
\[ \psi(\alpha_1) + \psi(\alpha_2) - \psi(\alpha_3) - \psi(\alpha_4) \in \Sigma.\]

With parameters as in the statement of Proposition \ref{sec-5-main}, consider the map $\pi : [-M, M] \rightarrow \Z$ given by 
\begin{equation}\label{pi-def} \pi(m) := \psi(\theta m).\end{equation} Since the map $m \mapsto \theta m$ is a homomorphism from $\Z$ to $\R$, we see that $\pi$ also has an almost-homomorphism property, namely that if $m_1 + m_2 = m_3 + m_4$ then 
\begin{equation}\label{almost-hom} \pi(m_1) + \pi(m_2) - \pi(m_3) - \pi(m_4) \in \Sigma.\end{equation}

Denote by $\mathscr{M}$ the set of all $m \in [-M, M]$ such that $\w_n(\theta m) \leq r$. Thus, by the assumptions of Proposition \ref{sec-5-main}, $|\mathscr{M}| \geq \eta M$. 
Denote $A := \pi(\mathscr{M})$. By the definition \eqref{pi-def} of $\pi$, \eqref{w-d} and the definition of $\mathscr{M}$, we see that $\dd_b(a) \leq r$ for all $a \in A$. For $a \in A$, denote by $X_a := \pi^{-1}(a) \cap \mathscr{M}$ the $\pi$-fibre above $a$. 
Decompose $A$ according to the dyadic size of these fibres, thus for $j \in \Z_{\geq 0}$ set 
\begin{equation}\label{fibre-sizes} A_j := \{ a \in A : 2^{-j-1} M < |X_a| \leq 2^{-j} M\}.\end{equation} Denote by $\mathscr{M}_j \subset \mathscr{M}$ the points of $\mathscr{M}$ lying above $A_j$, that is to say $\mathscr{M}_j := \bigcup_{a \in A_j} X_{a}$. Define $\eta_j$ by $|\mathscr{M}_j| = \eta_j M$. Since $\mathscr{M}$ is the disjoint union of the $\mathscr{M}_j$, we have
\begin{equation}\label{sum-eta-j} \sum_j \eta_j \geq \eta.\end{equation}
 By \eqref{fibre-sizes} we have $2^{-j-1} M |A_j| \leq |\mathscr{M}_j| \leq 2^{-j} M |A_j|$, and so 
\begin{equation}\label{a-size} 2^j \eta_j \leq  |A_j|  \leq 2^{j+1} \eta_j.\end{equation}

Now by a simple application of the Cauchy-Schwarz inequality any subset of $[-M, M]$ of size at least $\eps M$ has at least $\eps^4 M^3/4$ additive quadruples. In particular, for any $j \in \Z_{\geq 0}$ there are $\geq \eta_j^4 M^3/4$ additive quadruples in $\mathscr{M}_j$. By \eqref{almost-hom}, there is some $\sigma_j \in \Sigma$ such that, for $\geq 2^{-10} b^{-1}\eta_j^4 M^3$ additive quadruples in $\mathscr{M}_j$, we have 
\begin{equation}\label{pim-add} \pi(m_1) + \pi(m_2) = \pi(m_3) + \pi(m_4) + \sigma_j.\end{equation}

For each $j$, fix such a choice of $\sigma_j$. Now the number of such quadruples with $\pi(m_i) = a_i$ for $i = 1,2,3,4$ is, for a fixed choice of $a_1,\dots, a_4$ satisfying
\begin{equation}\label{ai-cond} a_1 + a_2 = a_3 + a_4 + \sigma_j,\end{equation}
the number of additive quadruples in $X_{a_1} \times X_{a_2} \times X_{a_3} \times X_{a_4}$, which is bounded above by $|X_{a_1}| |X_{a_2}| |X_{a_3}| \leq 2^{-3j} M^3$ since three elements of an additive quadruple determine the fourth. It follows that the number of $(a_1, a_2, a_3, a_4) \in A_j^4$ satisfying \eqref{ai-cond} is $\geq 2^{-10} b^{-1} 2^{3j} \eta_j^4$. By \eqref{a-size}, this is $\geq 2^{-13} b^{-1} \eta_j |A_j|^3$. 

Now if $S_1, S_2, S_3, S_4$ are additive sets then $E(S_1, S_2, S_3, S_4)$, the number of solutions to $s_1 + s_2 = s_3 + s_4$ with $s_i \in S_i$, is bounded by $\prod_{i = 1}^4 E(S_i)^{1/4}$, where $E(S_i)$ is the number of additive quadruples in $S_i$. This is essentially the Gowers-Cauchy-Schwarz inequality for the $U^2$-norm; it may be proven by two applications of Cauchy-Schwarz or alternatively from H\"older's inequality on the Fourier side. Applying this with $S_1 = S_2 = S_3 = A_j$ and $S_4 = A_j + \sigma_j$, and noting that $E(A_j + \sigma_j) = E(A_j)$, we see that $E(A_j) \geq 2^{-13} b^{-1}\eta_j |A_j|^3$.

By Proposition \ref{add-digit-hamming}, we have $|A_j| \leq 2^{4r + 13} b^{4r+1} \eta_j^{-1}$. Comparing with \eqref{a-size} gives $\eta_j \leq 2^{2r + 7-j/2} b^{2r + 1/2}$. Take $J$ to be the least integer such that $2^{J/2} \geq 2^{2r + 9} b^{2r+1/2} \eta^{-1}$; then $\sum_{j \geq J} \eta_j  < \eta$, and so by \eqref{sum-eta-j}, some $\mathscr{M}_j$, $j \leq J-1$, is nonempty. In particular, by \eqref{fibre-sizes} there is some value of $a$ such that $|X_a| \geq 2^{-J} M \geq 2^{-4r - 20} b^{-4r-1} \eta^2 M$. Fix this value of $a$ and set $\mathscr{M}' := X_a$. Thus, to summarise,
\begin{equation}\label{m-prime-lower} |\mathscr{M}'| \geq 2^{-4r - 20} b^{-4r-1} \eta^2 M\end{equation} and if $m \in \mathscr{M}'$ then $\pi(m) = a$. Note that the condition on $M$ in the statement of Proposition \ref{sec-5-main} implies (comfortably) that $|\mathscr{M}'| \geq 2$.

Note that, by \eqref{pi-basic-prop} and the definition \eqref{pi-def} of $\pi$, we have that if $m \in \mathscr{M}'$ then \begin{equation}\label{theta-ma-close} \Vert \theta m - b^{1-n} a \Vert \leq \frac{3}{4} b^{1-n}.\end{equation} 
Pick some $m_0 \in \mathscr{M}'$, and set $\mathscr{M}'' := \mathscr{M}' - m_0 \subset [-2M, 2M]$. By the triangle inequality and \eqref{theta-ma-close}, we have 
\begin{equation}\label{theta-m-bd} \Vert \theta m \Vert \leq \frac{3}{2} b^{1-n} < 2b N^{-1}\end{equation} for all $m \in \mathscr{M}''$. (Recall that, by definition, $N = b^n$.) Replacing $\mathscr{M}''$ by $-\mathscr{M}''$ if necessary (and since $|\mathscr{M}''| \geq 2$) it follows that there are at least $2^{-4r - 22} b^{-4r-1} \eta^2 M$ integers $m \in \{1,\dots, 2M\}$ satisfying \eqref{theta-m-bd}.

Now we apply Lemma \ref{vino-lemma}, taking $L = 2M$, $\delta_1 = 2b N^{-1}$ and $\delta_2= 2^{-4r - 22} b^{-4r-1} \eta^2$ in that result. The conditions of the lemma hold under the assumptions that $M, N \geq b^{20r} \eta^{-2}$ (using here the fact that $b \geq 3$). The conclusion implies that there is some positive integer $q \leq b^{20r} \eta^{-2}$ such that  $\Vert \theta q \Vert \leq b^{20r} \eta^{-2} N^{-1} M^{-1}$, which is what we wanted to prove.\end{proof}

Finally, we are in a position to prove Proposition \ref{weyl-restricted}, whose statement we recall now.
\begin{weyl-restrict-rpt}
Suppose that $k \geq 2$ and $b \geq 3$. Set $B := b^{6k^2}$. Suppose that $\delta \in (0,1)$ and that $k \mid n$. Suppose that $|\widehat{\mu_n}(\theta)| \geq \delta$, and that $N \geq (2/\delta)^{B}$, where $N := b^n$. Then there is a positive integer $q \leq (2/\delta)^{B}$ such that $\Vert \theta q \Vert \leq (2/\delta)^{B}N^{-k}$.
\end{weyl-restrict-rpt}
\begin{proof}
First apply Proposition \ref{sec4-main}. The conclusion is that for at least $(\delta/2)^{C}N^{k-1}$ values of $m$, $ |m| \leq C N^{k-1}$, we have $\tilde\w_n(\theta m) \leq C\log(2/\delta)$, where $C := b^{7k^2/2}$. By Lemma \ref{ww-tilde}, for these values of $m$ we have $\w_n(\theta m) \leq 16b^2 C \log(2/\delta)$. (For the definitions of $\tilde\w_n$ and $\w_n$, see Definitions \ref{wn-tilde-def} and \ref{wn-def} respectively.) Now apply Proposition \ref{sec-5-main} with $\eta := (\delta/2)^{C}C^{-1}$, $r = \lceil 16 b^2 C\log(2/\delta)\rceil$, $N = b^n$ (as usual) and $M := C N^{k-1}$. 

To process the resulting conclusion, note that $b^{20r} \eta^{-2} \leq (2/\delta)^{C'}$, with $C' := 2C + 320 b^2 C \log b + \log_2(C^2 b^{20}) < 321 b^2C \log b < b^{8}C < B$. Proposition \ref{weyl-restricted} then follows.\end{proof}

\appendix

\section{Box norm inequalities}\label{appA}

In this appendix we prove an inequality, Proposition \ref{box-gvn}, which is in a sense well-known: indeed, it underpins the theory of hypergraph regularity \cite{gowers-hypergraph} and is also very closely related to generalised von Neumann theorems and the notion of Cauchy-Schwarz complexity in additive combinatorics. 
We begin by recalling the basic definition of Gowers box norms as given in \cite[Appendix B]{green-tao-linearprimes}. 

\begin{definition}\label{defA1}
Let $(X_i)_{i \in I}$ be a finite collection of finite non-empty sets, and denote by $X_{I} := \prod_{i \in I} X_i$ the Cartesian product of these sets. Let $f : X_{I} \rightarrow \C$ be a function. Then we define the (Gowers-) box norm $\Vert f \Vert_{\Box(X_I)}$ to be the unique nonnegative real number such that 
\[ \Vert f \Vert_{\Box(X_I)}^{2^{|I|}} = \E_{x_I^{(0)}, x_I^{(1)} \in X_I} \prod_{\omega_I \in \{0,1\}^I} \mathcal{C}^{|\omega_I|} f(x_I^{(\omega_I)}).\] 
Here, $\mathcal{C}$ denotes the complex conjugation operator, and for any $x_I^{(0)} = (x_i^{(0)})_{i \in I}$ and $x_I^{(1)} = (x_i^{(1)})_{i \in I}$ in $X_I$ and $\omega_I = (\omega_i)_{i \in I} \in \{0,1\}^I$ we write $x_I^{(\omega_I)} = (x_i^{(\omega_i)})_{i \in I}$ and $|\omega_I| := \sum_{i \in I} |\omega_i|$.
\end{definition}
It is not obvious that $\Vert f \Vert_{\Box(X_I)}$ is well-defined, but this is so: see \cite[Appendix B]{green-tao-linearprimes} for a proof. Another non-obvious fact, whose proof may also be found in \cite[Appendix B]{green-tao-linearprimes}, is that $\Vert f \Vert_{\Box(X_I)}$ is a norm for $|I| \geq 2$. When $|I| = 1$, say $I = \{1\}$, we have $\Vert f \Vert_{\Box(X_I)} = |\sum_{x_1 \in X_1} f(x_1)|$, which is only a seminorm.

To clarify notation, in the case $I = \{1,2\}$ we have $\Vert f \Vert_{\Box(X_{\{1,2\}})}^4 =$
\[  \E_{\substack{ x_1^{(0)}, x_1^{(1)} \in X_1 \\ x_2^{(0)}, x_2^{(1)} \in X_2}} f(x_1^{(0)}, x_2^{(0)}) \overline{f(x_1^{(0)}, x_2^{(1)}) f(x_1^{(1)}, x_2^{(0)})} f(x_1^{(1)}, x_2^{(1)}).\]
Here is the inequality we will need. The proof is simply several applications of Cauchy-Schwarz, the main difficulty being one of notation.

\begin{proposition}\label{box-gvn}
Suppose that notation is as in Definition \ref{defA1}. Suppose additionally that, for each $i \in I$, we have a $1$-bounded function $\Psi_i : X_I \rightarrow \C$ which does not depend on the value of $x_i$, that is to say $\Psi_i(x_I) = \Psi_i(x'_I)$ if $x_j = x'_j$ for all $j \neq i$. Let $f : X_I \rightarrow \C$ be a function. Then we have
\[ \big| \E_{x_I \in X_I} \big(\prod_{i \in I} \Psi_i(x_I) \big) f(x_I) \big| \leq \Vert f \Vert_{\Box(X_I)}.\]
\end{proposition}
\begin{proof}
We proceed by induction on $|I|$, the result being a tautology when $|I| = 1$. Suppose now that $|I| \geq 2$, and that we have already established the result for smaller values of $|I|$. Let $\alpha$ be some element of $I$, and write $I' := I \setminus \{\alpha\}$. By Cauchy-Schwarz, the $1$-boundedness of $\Psi_\alpha$, and the fact that $\Psi_\alpha$ does not depend on $x_\alpha$, we have
\begin{align*}
\big|  & \E_{x_I \in X_I} \big(\prod_{i \in I} \Psi_i(x_I) \big) f(x_I) \big| ^2  \\ & = \big| \E_{x_{I'} \in X_{I'} }\Psi_\alpha(x_{I}) \E_{x_\alpha \in X_{\alpha}} \big(\prod_{i  \in I'} \Psi_i(x_I) \big) f(x_I) \big|^2 \\ & \leq  \E_{x_{I'}  \in X_{I'}} \big| \E_{x_\alpha \in X_{\alpha}}(  \prod_{i \in I'} \Psi_i(x_I) ) f(x_I)\big|^2 \\ & = \E_{x_\alpha^{(0)}, x_\alpha^{(1)} \in X_\alpha}  \E_{x_{I'} \in X_{I'}} \big( \prod_{i \in I'} \Psi_i(x_{I '}, x_\alpha^{(0)}) \overline{\Psi_i(x_{I'}, x_\alpha^{(1)})}\big) \times \\ &\qquad\qquad \qquad\qquad\qquad\qquad\qquad\qquad \times  f(x_{I'}, x_\alpha^{(0)}) \overline{ f(x_{I'}, x_\alpha^{(0)})  }.
\end{align*}
For fixed $x_\alpha^{(0)}, x_\alpha^{(1)}$ we may apply the induction hypothesis (with indexing set $I'$) with $1$-bounded functions
\[ \tilde \Psi_i(x_{I'}) := \Psi_i(x_{I'}, x_\alpha^{(0)}) \overline{\Psi_i(x_{I'}, x_\alpha^{(1)})}\] and with
\[ \tilde f(x_{I'}) = f(x_{I'}, x_\alpha^{(0)}) \overline{ f(x_{I'}, x_\alpha^{(0)})  },\] noting that $\tilde\Psi_i$ does not depend on $x_i$.

This gives

\[ \big|  \E_{x_I \in X_I} \big(\prod_{i \in I} \Psi_i(x_I) \big) f(x_I) \big| ^2 \leq \E_{x_\alpha^{(0)}, x_\alpha^{(1)} \in X_\alpha} \Vert f(\cdot, x_\alpha^{(0)}) \overline{f (\cdot, x_\alpha^{(1)})} \Vert_{\Box(X_{I'})}.\]
By H\"older's inequality, it follows that 
\[ \big|  \E_{x_I \in X_I} \big(\prod_{i \in I} \Psi_i(x_I) \big) f(x_I) \big| ^{2^{|I|}} \leq \E_{x_\alpha^{(0)}, x_\alpha^{(1)} \in X_\alpha} \Vert f(\cdot, x_\alpha^{(0)}) \overline{f (\cdot, x_\alpha^{(1)})} \Vert_{\Box(X_{I'})}^{2^{|I|-1}}.\]
However, the right-hand side is precisely $\Vert f \Vert_{\Box(X_I)}^{2^{|I|}}$, and the inductive step is complete.
\end{proof}

\section{Sumsets of subsets of $\{0,1\}^n$}\label{appB}

In this appendix we provide some comments on Theorem \ref{bkmp-thm}, which seems to have a very complicated history. In the case $r = 2$ it is due to Woodall \cite{woodall}, and independently to Hajela and Seymour \cite{hajela-seymour}.

In the general case, Theorem \ref{bkmp-thm} is a consequence of the following real-variable inequality, which was conjectured in \cite{hajela-seymour}. 
\begin{proposition}\label{real-var}
Let $r \geq 2$ be an integer. Suppose that $1 \geq x_1 \geq x_2 \geq \cdots \geq x_r \geq 0$. Then 
\[ (x_1 \cdots x_r)^{\gamma} + (x_1 \cdots x_{r-1} (1 - x_r))^{\gamma} + \cdots + ((1 - x_1) \cdots (1 - x_r))^{\gamma} \geq 1,\] where $\gamma := r^{-1} \log_2(r+1)$.
\end{proposition}

The deduction of Theorem \ref{bkmp-thm} from Proposition \ref{real-var} is a straightforward `tensorisation' argument, but no details are given in either \cite{bkmp} or \cite{hajela-seymour}. For the convenience of the reader we give the deduction below, claiming no originality whatsoever.

Proposition \ref{real-var} (and hence Theorem \ref{bkmp-thm}) was established by Landau, Logan and Shepp \cite{lls}, and 3 years later but seemingly independently (and in a more elementary fashion) by Brown, Keane, Moran and Pearce \cite{bkmp}. A discussion of the history of these and related problems is given by Brown \cite{brown-inequalities} but this appears to overlook \cite{lls}.

Finally,  we note that a result which is weaker in the exponent than Theorem \ref{bkmp-thm}, but quite sufficient for the purpose of proving the qualitative form of Theorem \ref{mainthm}, follows by an iterated application of a result of Gowers and Karam \cite[Proposition 3.1]{gowers-karam}. This avoids the need for the delicate analytic inequality in Proposition \ref{real-var}. Let us also note that the context in which Gowers and Karam use this result is in some ways analogous to ours, albeit in a very different setting. 

\begin{proof}[Proof of Theorem \ref{bkmp-thm}, assuming Proposition \ref{real-var}]
As stated in \cite{bkmp}, one may proceed in a manner `parallel' to arguments in \cite{brown-moran}, specifically the proof of Lemma 2.6 there. 
We proceed by induction on $n$. First we check the base case $n = 1$. Here, one may assume without loss of generality that $A_1 = \cdots = A_s = \{0,1\}$ and $A_{s+1} = \cdots = A_r = \{1\}$ for some $s$, $0 \leq s \leq r$. The density of $A_1 + \cdots + A_r$ in $\{0,1,\dots, r\}$ is then $(s+1)/(r+1)$, whilst $\alpha_1 = \cdots = \alpha_s = 1$ and $\alpha_{s+1} = \cdots = \alpha_r = 1/2$. The inequality to be checked is thus $(s+1)/(r+1) \geq 2^{-(r-s)\gamma}$. However, taking $x_1 = \cdots = x_s = 1/2$ and $x_{s+1} = \cdots = x_r = 0$ in Proposition \ref{real-var} yields $(s+1) 2^{-s\gamma} \geq 1$. Since $2^{r\gamma} = r+1$, the desired inequality follows.

Now assume the result is true for $n - 1$. Let $A_i^0$ be the elements of $A_i$ with first coordinate zero, and $A_i^1$ the elements of $A_i$ with first coordinate 1. Suppose that $|A_i^0| = x_i |A_i|$, and without loss of generality suppose that $x_1 \geq x_2 \geq \cdots \geq x_r$. Then the sets $A_1^0 + \dots + A_j^0 + A_{j+1}^1 + \dots + A_r^1$, $j = 0,\dots, r$ are disjoint, since the first coordinate of every element of this set is $j$. 

It follows that 
\[
|A_1 + \cdots + A_r|  \geq \sum_{j = 0}^r |A_1^0 + \dots + A_j^0 + A_{j+1}^1 + \dots + A_r^1| .\]
Note that $A_i^0$ is a subset of a copy $\{0,1\}^{n-1}$ of density $2\alpha_i x_i$, and that $A_i^1$ is a subset of (a translate of) $\{0,1\}^{n-1}$ of density $2\alpha_i (1 - x_i)$. 

By the inductive hypothesis, \begin{align*} |A_1^0 +  \dots & + A_j^0 +  A_{j+1}^1 + \dots + A_r^1| \\ & \geq (2^r \alpha_1 \cdots \alpha_r x_1 \cdots x_j (1 - x_{j+1}) \cdots (1 - x_r))^{\gamma}  (r+1)^{n-1} \\ & = (r+1)^n (\alpha_1 \cdots \alpha_r)^{\gamma}\big(x_1 \cdots x_j (1 - x_{j+1}) \cdots (1 - x_r)\big)^{\gamma} .\end{align*}
Performing the sum over $j$ and applying Proposition \ref{real-var}, the result follows. \end{proof}

\section{A diophantine lemma}

The following is a fairly standard type of lemma arising in applications of the circle method and is normally attributed to Vinogradov. We make no attempt to optimise the constants, contenting ourselves with a version sufficient for our purposes in the main paper.

\begin{lemma}\label{vino-lemma}
 Suppose that $\alpha \in \R$ and that $L \geq 1$ is an integer. Suppose that $\delta_1, \delta_2$ are positive real numbers satisfying $\delta_2 \geq 32 \delta_1$, and suppose that there are at least $\delta_2 L$ elements $n \in \{1,\dots, L\}$ for which $\Vert \alpha n \Vert \leq \delta_1$. Suppose that $L \geq 16/\delta_2$. Then there is some positive integer $q \leq 16/\delta_2$ such that $\Vert \alpha q \Vert \leq \delta_1\delta_2^{-1} L^{-1}$.
\end{lemma}
\begin{proof} Write $S \subseteq \{1,\dots, L\}$ for the set of all $n$ such that $\Vert \alpha n \Vert \leq \delta_1$; thus $|S| \geq \delta_2 L$.
By Dirichlet's lemma, there is a positive integer $q \leq 4L$ and an $a$ coprime to $q$ such that $|\alpha - a/q| \leq 1/4Lq$. Write $\theta := \alpha - a/q$; thus
\begin{equation}\label{thets}  |\theta| \leq \frac{1}{4Lq}.\end{equation}
The remainder of the proof consists of ``bootstrapping'' this simple conclusion. First, we tighten the bound for $q$, and then the bound for $|\theta|$.

Suppose that $n \in S$. Then, by \eqref{thets}, we see that 
\begin{equation}\label{will-need} \big\Vert \frac{an}{q} \big\Vert \leq \delta_1 + \frac{1}{4q}.\end{equation}

Now we bound the number of $n \in \{1,\dots, L\}$ satisfying \eqref{will-need} in a different way. Divide $\{1,\dots, L\}$ into $\leq 1+ \frac{L}{q}$ intervals of length $q$. In each interval, $\frac{an}{q} \md{1}$ ranges over each rational $\md{1}$ with denominator $q$ precisely once. At most $2 q(\delta_1 + \frac{1}{4q}) + 1 < 2 (\delta_1 q + 2)$ of these rationals $x$ satisfy $\Vert x \Vert \leq \delta_1 + \frac{1}{4q}$. Thus the total number of $n \in \{1,\dots, L\}$ satisfying \eqref{will-need} is bounded above by $2 \big(\frac{L}{q} + 1\big) (\delta_1 q + 2) = 2\delta_1 L + 2 \delta_1 q + \frac{4L}{q} + 4$.
It follows that 
\begin{equation}\label{star-posts} 2\delta_1 L + 2 \delta_1 q + \frac{4L}{q} + 4 \geq \delta_2 L. \end{equation}
Using $\delta_2 \geq 32 \delta_1$, $q \leq 4L$ and $L \geq 16/\delta_2$, one may check that the first, second and fourth terms on the left are each at most $\delta_2 L/4$. Therefore \eqref{star-posts} forces us to conclude that $4L/q  > \delta_2 L/4$, and therefore $q \leq 16/\delta_2$, which is a bound on $q$ of the required strength.

Now we obtain the claimed bound on $\Vert \alpha q \Vert$. Note that, by the assumptions and the inequality on $q$ just established, we have $\delta_1 \leq \delta_2/32 \leq 1/2q$, and so if $n \in S$ then, by \eqref{will-need}, we have $\Vert an/q \Vert < 1/q$, which implies that $q | n$. That is, all elements of $S$ are divisible by $q$. It follows from this and the definition of $\theta$ that if $n \in S$ then $\Vert \theta n \Vert = \Vert \alpha n \Vert \leq \delta_1$.
However, since (by \eqref{thets}) we have $|\theta| \leq 1/4L q$, for $n \in \{1,\dots, L\}$ we have $\Vert \theta n \Vert = | \theta n|$.
Therefore
\begin{equation}\label{thetty} |\theta n| \leq \delta_1\end{equation} for all $n \in S$. Finally, recall that $S$ consists of multiples of $q$ and that $|S| \geq \delta_2 L$; therefore there is some $n \in S$ with $|n| \geq \delta_2 q L$. Using this $n$, \eqref{thetty} implies that $|\theta| \leq \delta_1/q\delta_2 L$, and so finally $\Vert \alpha q \Vert \leq |\theta q| \leq \delta_1/\delta_2 L$. This concludes the proof.
\end{proof}

\end{document}